\documentclass[11pt]{amsart}
\usepackage{geometry}                % See geometry.pdf to learn the layout options. There are lots.
\usepackage{graphicx}
\usepackage{amssymb}
\usepackage{epstopdf}
\usepackage{amsmath}
\usepackage{color}
 \usepackage{bbm}

\usepackage[all]{xy}
\xyoption{matrix}
\xyoption{arrow}

\newtheorem{thm}{Theorem}[subsection]
\newtheorem{lem}[thm]{Lemma}
\newtheorem{cor}[thm]{Corollary}
\newtheorem{prop}[thm]{Proposition}

\theoremstyle{definition}
\newtheorem{defn}[thm]{Definition}
\newtheorem{eg}[thm]{Example}

 \newtheorem{rem}[thm]{Remark}

\numberwithin{equation}{section}

\newcommand{\mat}[1]{\ensuremath{
\left[\begin{matrix}#1
\end{matrix}\right]
}}
 
\def\xrarrow{\xrightarrow} %right arrow {label on top}

\def\then{\Rightarrow}

\def\into{\hookrightarrow}%or \rightarrowtail
\newcommand{\cofib}{\rightarrowtail}

\def\onto{\twoheadrightarrow}
\def\-{\text{-}}

\def\<{\left<}
\def\>{\right>}

\DeclareMathOperator{\Tr}{Tr}%\newcommand{\Tr}{\text{Tr}}
\DeclareMathOperator{\Ext}{Ext}%
\DeclareMathOperator{\add}{add} % additive category generated by ...

\newcommand{\field}[1]{\mathbb{#1}}
\newcommand{\ZZ}{\ensuremath{{\field{Z}}}}
\newcommand{\CC}{\ensuremath{{\field{C}}}}
\newcommand{\RR}{\ensuremath{{\field{R}}}}

\newcommand{\NN}{\ensuremath{{\field{N}}}}

\newcommand{\double}{\ensuremath{^{(2)}}}  % double a category

\newcommand{\kk}{\ensuremath{{\mathbbm{k}}}}

\newcommand{\commentout}[1]{}

\newcommand{\Ainftyup}[1]{% ZA-infinity AR component point up
\put(0,0){\qbezier(0,0)(.5,0)(.9,0)}
\put(0,0){\qbezier(0,0)(.25,.25)(.45,.45)}
\put(.45,.45){\qbezier(0,0)(.25,-.25)(.45,-.45)}
\put(.35,.14){#1}
}

\newcommand{\Ainftydown}[1]{% ZA-infinity AR component point up
\put(0,0){\qbezier(0,0)(.5,0)(.9,0)}
\put(0,0){\qbezier(0,0)(.25,-.25)(.45,-.45)}
\put(.45,-.45){\qbezier(0,0)(.25,.25)(.45,.45)}
\put(.35,-.2){#1}
}

\newcommand{\Ainftyinfty}[1]{% ZA-infinity AR component point up
\put(0,0){\qbezier(0,0)(.25,.25)(.45,.45)}
\put(.45,.45){\qbezier(0,0)(.25,-.25)(.45,-.45)}
\put(0,0){\qbezier(0,0)(.25,-.25)(.45,-.45)}
\put(.45,-.45){\qbezier(0,0)(.25,.25)(.45,.45)}
\put(.35,-.04){#1}
}

\def\el{\ell}
\def\ll{\lambda}

\newcommand{\cC}{\ensuremath{{\mathcal{C}}}}

\newcommand{\cMF}{\ensuremath{{\mathcal{MF}}}}
\newcommand{\cF}{\ensuremath{{\mathcal{F}}}}

\newcommand{\cM}{\ensuremath{{\mathcal{M}}}}

\newcommand{\cP}{\ensuremath{{\mathcal{P}}}}

\newcommand{\cT}{\ensuremath{{\mathcal{T}}}}

\newcommand{\cX}{\ensuremath{{\mathcal{X}}}}

\def\a{\alpha}
\def\b{\beta}

\def\d{\partial}

\def\e{\epsilon}

\def\s{\sigma}

\def\t{\tau}
\def\th{\theta}

\def\ul{\underline}

\title{Cluster categories coming from cyclic posets}
\author{Kiyoshi Igusa}
\address{Department of Mathematics, Brandeis University, Waltham, MA 02454}\email{igusa@brandeis.edu}
 \thanks{The first author is supported by NSA Grant \#H98230-13-1-0247}

\author{Gordana Todorov}
\address{Department of Mathematics, Northeastern University, Boston, MA 02115}
\email{g.todorov@neu.edu}
\thanks{The second author is supported by NSF Grant \#DMS-1103813}

\begin{document}

\begin{abstract} Cyclic poset are generalizations of cyclically ordered sets. In this paper we show that any cyclic poset gives rise to a Frobenius category over any discrete valuation ring $R$. The continuous cluster categories of \cite{IT09} are examples of this construction. If we twist the construction using an admissible automorphism of the cyclic poset, we generate other examples such as the $m$-cluster category of type $A_\infty$ ($m\ge3$).
\end{abstract}

\maketitle

\tableofcontents
 
%-----------------------------------------------------------------------------------------
% 		INTRODUCTION
%-----------------------------------------------------------------------------------------

\section*{Introduction}

In this paper we combine two concepts which have appeared in recent liturature: $\NN$-categories and matrix factorization categories in order to construct various triangulated categories including cluster categories of type $A$, continuous cluster categories and $m$-cluster categories of type $A_\infty$ for $m\ge3$. Matrix factorization was introduced by Eisenbud \cite{Eisenbud} and developed by Buchweitz in an unpublished paper \cite{Buchweitz} (see also \cite{BEH}), Orlov \cite{Orlov04}, \cite{Orlov12} and many others \cite{KhoRoz1}, \cite{Dycker}. $\NN$-categories are a version of the $\ZZ$-categories considerer by Drinfeld \cite{D} who, in turn, attributes it to Besser \cite{B} and Grayson \cite{G}. We take the multiplicative version which we call a $t^\NN$-category, very similar to a construction which occurs in van Roosmalen \cite{vanRoo12}. We call the basic underlying structure a ``cyclic poset.''

Given any set $X$, we define a cyclic partial ordering of $X$ to be an equivalence class of posets $\tilde X$ together with a free $\ZZ$ action and a bijection $\tilde X/\ZZ\cong X$ with the additional property of being \emph{recurrent}, i.e. that any $x,y\in X$ have liftings $\tilde x,\tilde y\in\tilde X$ so that $\tilde x<\tilde y$ (Definition \ref{defn of covering poset}). We observe that cyclic posets are special cases of $\NN$-categories (Definition \ref{def: N-categories}, taken essentially from \cite{D}, \cite{vanRoo12}). But we only consider these special cases. %Then we take the completed linearization of such a category over any discrete valuation ring $R$.

Given a discrete valuation ring $R$, we define a \emph{$t^\NN$-category} $\cP$ over $R$ to be a small $R$-category with two properties (Definition \ref{defn of nesto}). The first is that $\cP(x,y)\cong R$ for any two elements of the set of objects $X$ of $\cP$. The second is that $\cP(x,y)$ has a generator $f_{xy}$ with the property that, for any three $x,y,z\in X$,
\[
	f_{yz}\circ f_{xy}=t^n f_{xz}
\]
where $n=c(xyz)$ is a nonnegative integer. We call $f_{xy}$ the ``basic morphism'' from $x$ to $y$. When $R=\kk[[t]]$ is the power series ring in one variable over a field $\kk$, the category $\cP(X)$ can also be interpretted as the completed linearization of $X$ over $\kk$: $\cP(X)\cong\widehat{\kk X}$ (Proposition \ref{prop: completed linearization is isomorphic to RX}).

The function $c:X^3\to \NN$ is necessarily a reduced (2)-cocycle where ``reduced'' means $c(xxy)=0=c(xyy)$ for any $x,y\in X$ and the cocycle condition, equivalent to associativity of composition in $\cP$, is
\[
	c(xyz)-c(wyz)+c(wxz)-c(wxy)=0
\]
for any $w,x,y,z\in X$. Conversely, any reduced cocycle $c:X^3\to \NN$ on any set $X$ defines a $t^\NN$-category denoted $\cP(X,c)$ or simply $\cP(X)$. The structure of a (recurrent) cyclic partial ordering on any set $X$ is also determined by a uniquely determined reduced cocycle $c:X^3\to\NN$. Therefore, we can formally define a cyclic poset to be a pair $(X,c)$ where $c$ is a reduced cocycle on $X$.

Suppose $\phi$ is an admissible automorphisms $(X,c)$ (e.g., $\phi=id_X$: see Subsection \ref{subsection: admissible automorphism}). Then $\phi$ extends to an $R$-linear automorphism of $\cP(X)$ and there is a natural transformation $\eta_V:V\to \phi(V)$ given on each component of $V$ by the basic morphism $x\to\phi(x)$. Let $\cF_\phi(X)$ denote the category of all pairs $(V,d)$ where $V$ is an object of $\cP(X)$ and $d$ is an endomorphism of $V$ satisfying the following two properties.
\begin{enumerate}
\item $d^2=t$ is multiplication by $t$.
\item $d$ factors through $\eta_V:V\to \phi(V)$.
\end{enumerate}
If $\phi$ is admissible, then we show that $\cF_\phi(X)$ is a Frobenius category. Let $\cC_\phi(X)$ be the stable category of $\cF_\phi(X)$. Then $\cC_\phi(X)$ is a triangulated category.

The category $\cF_\phi(X)$ is a matrix factorization category, not by definition but by our theorem which says that any object $(V,d)$ in $\cF_\phi(X)$ decomposes in $\cP(X)$ as a direct sum of two projective objects $V=V_0\oplus V_1$ so that $d:V\to V$ has the form
\[
%\xymatrixrowsep{10pt}\xymatrixcolsep{10pt}
\xymatrix{%begin xy matrix
d=\mat{0&\b\\\a&0}: & V_0 \ar@/_1pc/[r]^{\a} & 
	V_1\ar@/_1pc/[l]_{\b} 
	}%end xy matrix
\]
In other words, $\a:V_0\to V_1,\b:V_1\to V_0$ are morphisms so that $\a\b=t$ and $\b\a=t$. 
In the analogous case when $V_0,V_1$ are replaced by free modules over a commutative ring $R$, such morphisms are given by square matrices $A,B$ so that $AB=BA=tI_n$, i.e., $(A,B)$ is a matrix factorization of $t$ in the traditional sense. In a matrix factorization category, one usually assumes that objects are $\ZZ/2$-graded. For us, the grading is undefined and there is no difference between even and odd morphisms.

In the remainder of the paper we explore examples and special cases of Frobenius categories of the form $\cF_\phi(X)$ and their stable categories $\cC_\phi(X)$. There include:
\begin{enumerate}
\item Continuous Frobenius categories over $R$
\item Cluster categories of type $A_n$ over the field $\kk=R/(t)$ where $1\le n\le\infty$
\item $m$-cluster categories of type $A_\infty$ over $\kk$ for $m\ge3$
\end{enumerate}

This paper was partially written while the second author was at the Mathematical Sciences Research Institute (MSRI) in Berkeley. 

\section{General theory}

%-----------------------------------------------------------------------------------------
%  	A1:  Cyclic posets
%-----------------------------------------------------------------------------------------

\subsection{Cyclic posets}

%This structure can also be described using a \emph{cocycle}:

%Given three objects, $x,y,z\in X$, we have
%\[	f_{yz}f_{xy}=nf_{xz}\]
%for some $n\in\NN$. Let $c:X^3\to\NN$ be the function $c(xyz)=n$ (viewing $X^3$ as the set of all 3-letter words with letters in $X$).

%\begin{prop}
%$c$ is a cocycle:
%\[	c(xyz)-c(wyz)+c(wxz)-c(wxy)=0\]
%which is reduced: $c(xxy)=0=c(xyy)$. Furthermore, any reduced cocycle comes from some cyclic poset $X$.\end{prop}

%-----------------------------------------------------------------------------------------
%  	A1:  Cyclic posets, part 2
%-----------------------------------------------------------------------------------------

We will carefully go over the ``covering poset'' description of a cyclic poset structure on a set $X$ and show that this cyclic poset structure is equivalent to choosing a reduced cocycle $c:X^3\to\NN$. First, we recall that a \emph{poset} is a set $X$ with a transitive, reflexive relation $\le$. Two elements of $X$ are \emph{equivalent} and we write $x\approx y$ if $x\le y$ and $y\le x$. We write $x<y$ if $x\le y$ and $x\not\approx y$. A morphism of posets $f:X\to Y$ is a mapping which preserves the relation: $x\le y\then f(x)\le f(y)$.

\subsubsection{\ul{Covering posets}}

\begin{defn}\label{defn of covering poset}
A \emph{covering poset} of set $X$ is a poset $\tilde X$ together with a surjective mapping $\pi:\tilde X\to X$ satisfying the following.
\begin{enumerate}
\item The inverse image $\pi^{-1}(x)\subseteq \tilde X$ of any $x\in X$ is isomorphic as a poset to the set of integers with the usual ordering.
\item There exists an automorphism $\s$ of the poset $\tilde X$ so that, for any $\tilde x\in \tilde X$, $\s \tilde x$ is the smallest element of $\pi^{-1}(\pi(\tilde x))$ which is greater than $\tilde x$. %Then $\s$ is an automorphism of poset $\tilde X$, i.e., $\tilde x<\tilde y$ iff $\s\tilde x<\s\tilde y$.
\item For every pair $\tilde x,\tilde y\in\tilde X$ there exists an integer $m\ge0$ so that $\tilde x\le\s^m\tilde y\le\s^{2m}\tilde x$.
\end{enumerate}
An \emph{equivalence} of covering posets $f:\tilde X\to \tilde X'$ is a $\s$-equivariant poset isomorphism over $X$, i.e., $\pi'f=\pi$ and $\s f=f\s$.
\end{defn}

\begin{rem}
In the above definition, the automorphism $\s$ of the covering poset $\tilde X$ is uniquely determined.
\end{rem}

%\begin{rem}Since $\s$ acts freely on $\tilde X$ it also acts freely on the Hasse diagram of $\tilde X$. The quotient under this action is a diagram whose vertices are orbits of the action of $\s$. Using the bijection $\tilde X/\<\s\>\cong X$, this diagram has vertex set $X$. In this diagram the vertical lines in the Hasse diagram must be drawn as arrows since the diagram has oriented cycles.\end{rem}

\begin{eg}
The simplest example of this is a set with $n$ elements $Z_n=\{1,2,\cdots,n\}$. One covering poset of $Z_n$ is given by the set of integers $\tilde Z_n=\ZZ$ with the usual ordering and the automorphism $\s x=x+n$:
$
	\tilde X:\quad \cdots <1<2<\cdots<n<\s 1<\s 2<\cdots
$. 
%\noi Another covering poset for the same set $X=\{1,2,\cdots,n\}$ is a connected translation quiver with translation $\t$ in which there are $n$ orbits of the $\t$ action, all doubly infinite.
\end{eg}

\subsubsection{\ul{Cocycles}}

We will show that covering posets over $X$ are classified by reduced cocycles $c:X^3\to\NN$ as defined below.

\begin{defn}
A \emph{cocycle} on $X$ is a function $c:X^3\to\NN$, written $c(xyz)$, so that
\[
	\delta c(wxyz):=c(xyz)-c(wyz)+c(wxz)-c(wxy)=0
\]
for all $w,x,y,z\in X$. We say that $c$ is \emph{reduced} if $c(xxy)=c(xyy)=0$ for all $x,y\in X$.
\end{defn}

\begin{eg}
Central group extensions $\ZZ\to E\to G$ are classified by elements of $H^2(G;\ZZ)$ and represented by factor sets $f:G^2\to \ZZ$ which may be taken to be reduced in the sense that $f(g,h)=0$ if $g=1$ or $h=1$. If this reduced factor set happens to have only nonnegative values, it gives a reduced cocycle $c$ by the formula $c(xyz)=f(x^{-1}y,y^{-1}z)$. Every such cocycle gives a distinct cyclic poset structure on the underlying set of $G$.
\end{eg}

\begin{eg}
A \emph{cyclic ordering} on a set $X$ is the same as a reduced cocycle on $X$ with values $0,1$ where $c(xyz)=0$ if and only if either $x,y,z\in X$ are in cyclic order with $x\not\approx z$ or $x\approx y\approx z$.
\end{eg}

\begin{defn}
Given a covering poset $\pi:\tilde X\to X$ and a section $\ll:X\to \tilde X$, we define the corresponding \emph{distance function} $b_\ll:X^2\to \ZZ$ by letting $b_\ll(xy)=m$ be the smallest integer so that $\ll (x)\le\s^m\ll (y)$.
\end{defn}

We note that $b_\ll$ is \emph{reduced} in the sense that $b_\ll(xx)=0$ for all $x\in X$.

%Given a set $X$ and a covering poset $\pi:\tilde X\to X$, we define a function $c:X^3\to\NN$ as follows. First, choose a lifting $\tilde x\in\pi^{-1}(x)$ for every $x\in X$. Then we have a function
%\[b:X^2\to\ZZ\]
%given by letting $b(x,y)=j$ be the smallest integer so that $\tilde x\le \s^{j}\tilde y$. Such an integer exists and lies between $m$ and $-m$ where $m$ is given by Condition (3) in the definition above. (To see that $j\ge-m$, note that $\tilde x>\s^{-1}\tilde x\ge \s^{-1-m}\tilde y$ by (1),(2) and (3).) We call $b:X^2\to\ZZ$ the \emph{distance function}. This is a {reduced cochain} which depends on the choice of liftings $\tilde x$. (\emph{Reduced} means $b(x,x)=0$ for all $x\in X$.) As an example, the distance function for $Z_n$ with the given choice of liftings is $b(xy)=0$ if $x\le y$ and $b(xy)=1$ if $x>y$. (For simplicity we suppress the commas.)

%If $b(xy)=j$ and $b(yz)=k$ then $\tilde x\le \s^j\tilde y\le \s^{j+k}\tilde z$. Therefore, $b(xy)+b(yz)\ge b(xz)$. So, we can define the \emph{structure cocycle} $c:X^3\to \NN$ to be the coboundary of the distance function $b$:
%\[c(xyz)=\delta b(xyz)=b(xy)+b(yz)-b(xz).\]
%Since $b$ is a reduced cochain, the function $c:X^3\to\NN$ is a reduced cocycle as defined in the introduction.% This also follows from the observation that $c$ is, by definition, the coboundary of the reduced $1$-cochain $m:X^2\to\ZZ$ given by $m(xy)=j$, the smallest integer so that $\tilde x\le\s^j\tilde y$.

\begin{lem} Let $\pi:\tilde X\to X$ be a covering poset and $\ll$ a section of $\pi$. Then
\[
	c(xyz):=\delta b_\ll(xyz)=b_\ll(yz)-b_\ll(xz)+b_\ll(yz)
\]
is a reduced cocycle $c:X^3\to\NN$ and $c$ is independent of the choice of sections $\ll:X\to \tilde X$.
\end{lem}

\begin{proof} Since $\delta^2=0$, it follows that $c$ is a cocycle. Also, $c$ is easily seen to be reduced. To prove uniqueness, suppose that $\ll'$ is another section of $\pi$. Then $\ll'(x)=\s^{a(x)}\ll(x)$ for some function $a:X\to \ZZ$. Then the two distance function are related by $b_{\ll'}(xy)=b_\ll(xy)-a(y)+a(x)$ or $b_{\ll'}=b_\ll-\delta a$. So, $\delta b_{\ll'}=\delta b_\ll-\delta^2a=\delta b_\ll$ as claimed.
\end{proof}

\begin{rem}
Since $c$ is uniquely determined, we refer to it as the \emph{structure cocycle} of the covering poset $\pi:\tilde X\to X$
\end{rem}

We say that $x,y\in X$ are \emph{equivalent} and we write $x\approx y$ if $c(xyx)=c(yxy)=0$. It is easy to see that this is an equivalence relation and that $x\approx y$ if and only if $\tilde x\approx \tilde y$ for some liftings $\tilde x,\tilde y$ of $x,y$ to $\tilde X$.

\begin{lem}
Given any reduced cocycle $c:X^3\to\NN$ on any set $X$, there exists a covering poset $\tilde X\to X$ whose structure cocycle is $c$. Furthermore, $\tilde X$ is unique up to poset isomorphism over $X$.
\end{lem}

\begin{proof} Choose a base point $x_0\in X$. Then we can express $c$ as $c=\delta b$ where $b:X^2\to \ZZ$ is given by $b(xy)=c(x_0xy)$. A covering poset $\tilde X$ can now be given as follows. As a set, let $\tilde X=X\times \ZZ$ with $\s$ action given by $\s(x,i)=(x,i+1)$. Then the partial ordering of $\tilde X$ is given by: $(x,j)\le (y,k)$ if $k\ge j+b(xy)$. In other words, $b(xy)$ is the distance function for the lifting $\tilde x=(x,0)$. Therefore the structure cocycle of $\tilde X$ is $\delta b=c$.

If $\tilde X'\to X$ is another covering poset with structure cocycle $c$ then a $\s$-equivariant poset isomorphism $f:X\times\ZZ\to \tilde X'$ is given as follows. Choose a fixed lifting $\tilde x_0\in \tilde X'$. Then for each $x\in X$ let $\tilde x\in \tilde X'$ be the unique lifting of $x$ so that $\tilde x_0\le \tilde x$ but $\s\tilde x_0\not\le \tilde x$. Then, for any $x,y\in X$, the smallest integer $j$ so that $\tilde x\le \s^j\tilde y$ is $j=c(x_0xy)=b(xy)$. Therefore, the covering posets $\tilde X'$ and $X\times \ZZ$ have the same distance function which implies that they are isomorphic with isomorphism $f:X\times \ZZ\to \tilde X'$ given by $f(x,j)=\s^j\tilde x$.% $f$ is also evidently $\s$-equivariant.
\end{proof}

This proves the following theorem.

\begin{thm} For any set $X$, there is a 1-1 correspondence between equivalence classes of covering posets over $X$ and reduced cocycles $c:X^3\to\NN$.
\end{thm}

Intuitively, a cyclic poset structure on a set $X$ is an equivalence class of covering posets $\tilde X\to X$. However, because of the above theorem, we use the following simpler definition.

\begin{defn}
A \emph{cyclic poset} is defined to be a pair $(X,c)$ where $X$ is a set and $c:X^3\to\NN$ is a reduced cocycle on $X$.
\end{defn}

\begin{eg} The \emph{product} of cyclic posets $X_1\times X_2$ has reduced cocycle $c(xyz):=c_1(x_1y_1z_1)+c_2(x_2y_2z_2)$ where $c_1,c_2$ are the cocycles of $X_1,X_2$ and $x=(x_1,x_2)$, etc.
 \end{eg}

\begin{eg}\label{cyclic poset XP}
Suppose that $X$ is a cyclic poset with covering poset $\tilde X$ and $P$ is another poset. Then another cyclic poset $X\ast P$ with underlying set $X\times P$ can be constructed as follows. The covering poset $\widetilde{X\ast P}$ of $X\ast P$ is the Cartesian product $\tilde X\times P$ with lexicographic order. So, $(\tilde x,p)\le (\tilde y,q)$ if either $\tilde x<\tilde y$ or $\tilde x\approx \tilde y$ and $p\le q$. The $\s$ action is given only on the first coordinate: $\s(\tilde x,p)=(\s \tilde x,p)$. An important example is $P=\ZZ$. If $X$ is cyclically ordered, then so is $X\ast\ZZ$.
\end{eg}

\subsubsection{$\NN$-categories}

Following, van Roosmalen 1011.6077, p.10 and Drinfeld 0304064, p.5,  (who refers to Besser and Greyson), we note that a cyclic poset structure on a set $X$ is equivalent to a special case of an $\NN$-category structure on (the object set) $X$.
 
\begin{defn}\label{def: N-categories}
An \emph{$\NN$-category} is a category $\cX$ with the property that the additive monoid $\NN$ acts freely on every Hom set
\[
	\NN\times \cX(x,y)\to \cX(x,y)%=\NN f_{yx}=\{f_{yx}=0f_{yx},1f_{yx},2f_{yx},\cdots \}
\]
so that composition satisfies:
\[
	nf\circ mg=(n+m)fg:x\to z
\]
(\emph{Acting freely} means Hom sets are disjoint unions of copies of $\NN$: $\cX(x,y)=\coprod \NN f_i$.)
\end{defn}

\begin{prop}
A cyclic poset structure on a set $X$ is the same as an $\NN$ category $\cX$ with object set $X$ so that every Hom set $\cX(x,y)$ is freely generated by one morphism $f_{xy}$.
\end{prop}
Note that, given three objects, $x,y,z\in X$, we have
\begin{equation}\label{def of cocycle}
	f_{yz}f_{xy}=nf_{xz}
	\end{equation}
for some $n\in\NN$.

%\newpage

%-----------------------------------------------------------------------------------------
%  	A2:  Linearization
%-----------------------------------------------------------------------------------------

\subsection{Linearization of cyclic posets} From now on we assume that $R$ is a discrete valuation ring with a fixed uniformizer $t$. We will construct the \emph{linearization} of any cyclic poset $X=(X,c)$ over $R$. %This will be an $R$-category, namely Hom sets will be $R$-modules and composition will be $R$-bilinear.

%-----------------------------------------------------------------------------------------
%  	A2:  Linearization, part 2:
%-----------------------------------------------------------------------------------------

%We define $t^\NN$-categories $\cP$, note that they are associated to cyclic posets $(X,c)$ and then show that $\cP$ and $(X,c)$ uniquely determine each other in a functorial way. 

\subsubsection{\ul{$t^\NN$-categories}}

\begin{defn}\label{defn of nesto}
A \emph{$t^\NN$-category} over $R$ is defined to be a small category $\cP$ with the following two properties:
\begin{enumerate}
\item $\cP(x,y)$ is a free $R$-module with one generator $f_{xy}$ for all $x,y$ in the object set $X$ of $\cP$. Thus $\cP(x,y)=\{rf_{xy}\,|\, r\in R\}$.
\item For any objects $x,y,z\in X$, there is a nonnegative integer $n$ so that for all $r,s\in R$,
\[%begin{equation}
	rf_{yz}\circ sf_{xy}=rst^nf_{xz}.
\]%end{equation}
\end{enumerate}
\end{defn}

Note that this is an \emph{$R$-category}: Hom sets are $R$-modules and composition is $R$-bilinear. Note that the integer $n$ in the above equation is uniquely determined since $t^n=t^m$ in $R$ if and only if $n=m$. So, $n$ is a well defined element of $\NN$ which depends only on $x,y,z$.

\begin{lem}
Given a $t^\NN$-category $\cP$ over $R$, there is a unique reduced cocycle $c:X^3\to\NN$ on its set of objects $X$ so that $f_{yz}f_{xy}=t^{c(xyz)}f_{xz}$ for all $x,y,z\in X$.
\end{lem}

\begin{proof} We have already remarked that $n$ gives a well-defined function $X^3\to\NN$ which we will denote by $c$. It remains to show that $c$ is a reduced cocycle.

First note that, if $rf_{xx}$ is the identity map on $x\in X$ then, for any $y\in X$, we must have $f_{xy}=f_{xy}(rf_{xx})=rt^{c(xxy)}f_{xy}$. Therefore, we must have $r=1$ and $c(xxy)=0$. Similarly, $c(xyy)=0$. So, $c$ is reduced. Associativity of composition gives 
%\[{c(xyz)+c(wxz)}={c(wxy)+c(wyz)}\text{  since}\]
\[
	(f_{yz}f_{xy})f_{wx}=t^{c(xyz)}\!f_{xz}f_{wx}=t^{c(xyz)+c(wxz)}\!f_{wz}\]
	\[
=f_{yz}(f_{xy}f_{wx})=t^{c(wxy)}\!f_{yz}f_{wy}=t^{c(wxy)+c(wyz)}\!f_{wz}.
\]
Therefore, $\delta c={c(xyz)+c(wxz)}-{c(wxy)-c(wyz)}=0$. Thus $c$ is a reduced cocycle. %Uniqueness of the function $c$ follows from the fact that $t^n= t^m$ in $R$ if and only if $n=m$.
\end{proof}

\begin{defn}
We call $(X,c)$ the \emph{underlying cyclic poset} of the $t^\NN$-category $\cP$.
\end{defn}

\begin{eg}\label{eg: tN category from distance function b}
Suppose that $b:X^2\to\NN$ is a distance function for $c:X^3\to\NN$ in the sense that $c=\delta b$. Let $\cP$ denote the $R$-category with object set $X$ and morphism sets $\cP(x,y)=(t^{b(xy)})$, the ideal in $R$ generated by $t^{b(xy)}$, with composition given by multiplication. Then composition of any morphisms $x\to y\to z$ will be divisible by $t^{b(xy)+b(yz)}=t^{c(xyz)+b(xz)}$ and therefore will lie in the ideal $(t^{b(xz)})$ as required. Then $\cP$ is a $t^\NN$-category with underlying cyclic poset $(X,c)$.
\end{eg}

By an \emph{isomorphism} of cyclic posets $(X,c)\cong (X',c')$ we mean a bijection $x\leftrightarrow x'$ so that $c(xyz)=c'(x'y'z')$ for all $x,y,z\in X$.

\begin{prop} (a) Given any cyclic poset $(X,c)$, there is a $t^\NN$-category $RX$ with underlying cyclic poset $(X,c)$.

(b) Any $t^\NN$-category $\cP$ is isomorphic to $RX$ if and only if its underlying cyclic poset is isomorphic to $(X,c)$.
\end{prop}

\begin{proof} (a) The $t^\NN$-category $RX$ is given as follows. We take $X$ to be the object set of $RX$ but we denote by $P_x$ the object in $RX$ corresponding to $x\in X$. For any $x,y\in X$ we take $RX(P_x,P_y)\cong R$ to be the free $R$ module generated by the single element $f_{xy}$. Composition of morphisms $P_x\to P_y\to P_z$ is given by $(rf_{yz})\circ (sf_{xy})=rst^{c(xyz)}f_{xz}$. As in the proof of the lemma above, composition is associative since $c$ is a cocycle and $f_{xx}$ is the identity map on $x$ since $c$ is reduced.

(b) Suppose we have a $t^\NN$-category $\cP$ with underlying cyclic poset $(X',c')$ and an isomorphism $\Phi:RX\to \cP$. Then, $\Phi$ gives a bijection between the object set $X$ of $RX$ and the object set $X'$ of $\cP$ which we denote $P_x\mapsto x'$. The basic morphism $f_{xy}:P_x\to P_y$ maps to a generator $r_{xy}f_{x'y'}\in \cP(x',y')$. So, $r_{xy}$ must be a unit in $R$. Since $\Phi$ sends the identity $f_{xx}$ on $P_x$ to the identity on $x'$, we must have $r_{xx}=1$. Finally, the relation $f_{yz}f_{xy}=t^{c(xyz)}f_{xz}$ gives the relation
\[
	r_{yz}r_{xy}f_{y'z'}f_{x'y'}=r_{xz}t^{c(xyz)}f_{x'z'}
\]
which implies that
\[
	r_{yz}r_{xy}t^{c'(x'y'z')}=r_{xz}t^{c(xyz)}
\]
Since the $r$'s are units, we conclude that $c'(x'y'z')=c(xyz)$ for all $x,y,z\in X$.

Conversely, if the underlying cyclic poset $(X',c')$ of $\cP$ is isomorphic to $(X,c)$ then we get a bijection $x'\leftrightarrow P_x$ on the set of objects and this extends to the morphism sets by mapping $f_{x'y'}$ to $f_{xy}$. The cyclic poset isomorphism implies that the composition laws correspond. Therefore, we have an isomorphism of categories $\cP\cong RX$.
\end{proof}

In the special case when $R=\kk[[t]]$, the $t^\NN$-category $RX$ has another description as the completed orbit category of $\kk\tilde X$. This idea comes from \cite{vanRoo12}. First, we recall that, since $\tilde X$ is a poset, it has a linearization $\kk\tilde X$ over $\kk$ defined as follows. The objects of $\kk\tilde X$ are the elements of $\tilde X$ and hom sets $\kk\tilde X(x,y)$ are either $\kk$ or $0$ depending on whether $x\le y$ or not. Composition of morphisms is given by multiplication by scalars. The $\s$-\emph{orbit category} of $\kk\tilde X$ is the category whose objects are the $\s$-orbits $\bigoplus \s^nx\in Add\,\kk\tilde X$ and whose morphisms are $\kk$-linear morphisms $f:\bigoplus \s^nx\to \bigoplus\s^my$ which are $\s$-\emph{equivariant} in the sense that $f\circ \s=\s\circ f$. This is the same as saying that the $(m,n)$ component of the morphism $f$ is equal to the $(m+i,n+i)$ component of $f$ for all $i$. Note that the $(m,n)$ component of $f$ can only be nonzero if $\s^my\ge\s^nx$.
%. It is not too hard to see that $Hom(\bigoplus \s^nx,\bigoplus \s^ny)\cong \kk[t]$ and composition is the product of the polynomials times $t^{c(xyz)}$.

\begin{defn}\label{defn: completed linearization}
We define the \emph{completed linearization} $\widehat{\kk X}$ of a cyclic poset $X$ with covering poset $\pi:\tilde X\to X$ to be the following category. The objects of $\widehat{\kk X}$ are the elements of $X$ and a morphism $f:x\to y$ is defined to be an infinite matrix $f=(f_{ba})$ with entries $f_{ba}\in\kk$ and indexed over all $a\in \pi^{-1}(x)$, $b\in \pi^{-1}(y)$ so that 
\begin{enumerate}
\item $f_{ba}=0$ if $b\not\ge a$.
\item $f_{ba}=f_{\s(b)\s(a)}$ for all $a,b$. 
\end{enumerate}
Composition of morphisms: $f\circ g=h$ is given by
\[
	h_{ca}=\sum_{b\in \pi^{-1}(y)}f_{cb}g_{ba}
\]
\end{defn}

\begin{prop}\label{prop: completed linearization is isomorphic to RX}
When $R=\kk[[t]]$, there is a natural isomorphism of categories $RX\cong \widehat{\kk X}$.
\end{prop}

\begin{proof}
To construct the isomorphism, we need to choose a fixed section $\ll:X\to\tilde X$ and let $b_\ll:X^2\to\ZZ$ be the corresponding distance function. Then, $c=\delta b_\ll$ and
\[
	\s^i\ll(x)\le \s^j\ll(y)\iff j-i- b_\ll(x,y)\ge0.
\]
An isomorphism $\Theta:RX\to \widehat{\kk X}$ is given as follows. On objects, $\Theta$ is the identity map: $\Theta x=x$ for all $x\in X$. On each morphism $g=\sum g_n t^n f_{xy}:x\to y$ in $RX$, $g_n\in\kk$, we let $\Theta g:x\to y$ be the morphism in $\widehat{\kk X}$ given by the matrix whose $(\s^j\ll(y),\s^i\ll(x))$-entry is $g_n$ where $n=j-i-b_\ll(x,y)$. Since this describes all possible morphisms $x\to y$ in $\widehat{\kk X}$, $\Theta$ gives a bijection on object sets and a linear isomorphism of Hom sets. It remains only to show that $\Theta$ is a functor.

To see this, take another morphism $h=\sum h_mt^m f_{yz}:y\to z$ in $RX$. Then the matrix of $\Theta h$ has $(\s^k\ll(z),\s^j\ll(y))$-entry equal to $h_m$ where $m=k-j-b_\ll(y,z)$. The product of these two matrices has $(\s^k\ll(z),\s^i\ll(x))$-entry equal to $\sum h_mg_n$ where the sum is over all $n,m$ so that
\begin{equation}\label{first equation for hg}
	n+m=k-i-b_\ll(x,y)-b_\ll(y,z)
\end{equation}
The composition of the morphisms $g,h$ in $RX$ is equal to
\[
	h\circ g=\sum h_mg_nt^{m+n}f_{yz}f_{xy}=\sum h_mg_nt^{m+n+c(xyz)}f_{xz}
\]
Then $\Theta(hg)$ has  $(\s^k\ll(z),\s^i\ll(x))$-entry equal to $\sum h_mg_n$ where the sum is over all $n,m$ so that
\[%begin{equation}\label{second equation for hg}
	n+m+c(xyz)=k-i-b_\ll(x,z)
\]%end{equation}
However, this is equivalent to \eqref{first equation for hg} since $c=\delta b_\ll$. Thus $\Theta$ commutes with composition. It also takes $id_x$ to $id_x$. So, $\Theta$ is a functor and therefore an isomorphism of categories.
\end{proof}

%\newpage

%-----------------------------------------------------------------------------------------
%  	A3:  Frobenius categories
%-----------------------------------------------------------------------------------------

\subsection{Frobenius category} For any cyclic poset $X$, let $\cP(X)=add\,RX$ denote the additive $R$-category generated by $RX$.

 \begin{defn}\label{defn: MF(X)}
 Let $\cMF(X)$ denote the category of all pairs $(P,d)$ where $P\in\cP(X)$ and $d:P\to P$ so that $d^2=\cdot t$ (multiplication by $t$). Morphism $f:(P,d)\to (Q,d)$ are maps $f:P\to Q$ so that $df=fd$.
 \end{defn}
 
The key property is the \emph{adjunction lemma}:

\begin{lem}
The functor $G:\cP(X)\to \cMF(X)$ given by
\[
%\xymatrixrowsep{10pt}\xymatrixcolsep{10pt}
\xymatrix{%begin xy matrix
GP:=\left(
	P\oplus P,\mat{0&t\\ 1&0}
	\right):& P \ar@/_1pc/[r]^{id} & 
	P\ar@/_1pc/[l]_{\cdot t} 
	}%end xy matrix
\]
is both left and right adjoint to the forgetful functor $F:\cMF(X)\to \cP(X)$.
\end{lem}

For the proof, see the twisted version below.

 \begin{thm}
 For any cyclic poset $X$, $\cMF(X)$ is a Frobenius category where a sequence \[
 (A,d)\cofib (B,d)\onto (C,d)\]
  is defined to be exact in $\cMF(X)$ if $A\cofib B\onto C$ is (split) exact in $\cP(X)$. Objects isomorphic to $GP$ for some $P$ in $\cP$ are the projective injective objects of $\cMF(X)$. %$\underline f=\underline g$ iff $f-g=es+se$ for some $s:P\to Q$.
 \end{thm}
 
The twisted version of this theorem is proved in the next section. The fact that $GP$ is projective and injective follows from the adjunction lemma.

\begin{eg}
The main example of this construction is the \emph{continuous Frobenius category} $\cMF(S^1)$ which comes from the cyclically ordered set $S^1$. The stable category is the \emph{continuous cluster category} $\ul \cMF(S^1)=\cC_\pi$. These are topological categories. The topology is used to define the cluster structure on $\cC_\pi$, namely, two objects $X,Y$ are \emph{compatible} if the ordered pair $(X,Y)$ lies in the closure of the set of all $(X',Y')$ satisfying $\Ext^1(X',Y')=0=\Ext^1(Y',X')$ and a \emph{cluster} is defined to be a maximal set of pairwise compatible indecomposable objects which is also a discrete set (every element lies in an open set containing no other objects of the cluster). See \cite{IT09}, \cite{IT10} for details.
\end{eg}

To obtain other kinds of examples, we need to take a twisted version $\cMF_\phi(X)$ of the Frobenius category.

%\newpage

%-----------------------------------------------------------------------------------------
%  	A4:  Twisted version
%-----------------------------------------------------------------------------------------

\subsection{Twisted version}\label{subsection: admissible automorphism} We need to modify the above construction using an admissible automorphism of a cyclic poset.

An automorphism $\phi$ of $(X,c)$ will be called \emph{admissible} if it is covered by a $\s$-equivariant order preserving bijection $\tilde \phi$ of the covering poset $\tilde X$ to itself satisfying the property:
 \[
 	x\le \tilde\phi x \le\tilde\phi^2 x <\sigma x
 \]
 for all $x\in\tilde X$. There is a corresponding additive $R$-linear automorphism of $\cP(X)$, which we also call $\phi$, defined on indecomposable objects by $\phi P_x=P_{\phi x }$ and on basic morphisms by $\phi f_{xy}=f_{\phi x, \phi y }$. In $\cP(X)$ the condition above gives morphisms
 \[
 	P_x\xrarrow{\eta_x}\phi P_x=P_{\phi x }\xrarrow{\xi_x}P_x
 \]
giving natural transformations 
 \[
 	P\xrarrow{\eta_P}\phi P\xrarrow{\xi_P}P
 \]
 whose composition $\xi_P\circ \eta_P:P\to P$ is multiplication by $t$. Since $t$ is not a zero divisor, neither are $\xi_P$ nor $\eta_P$. Therefore, the reverse composition $\eta_P\circ \xi_P:\phi P\to \phi P$ is also multiplication by $t$ since $\xi\eta \xi=t\xi=\xi t$ making $\xi(\eta\xi-t)=0$.
 
 \begin{defn}\label{defn: twisted MF-phi(X)}
 Let $\cMF_\phi(X)$ be the full subcategory of $\cMF(X)$ of all $(P,d)$ where $d:P\to P$ factors through $\eta_P:P\to\phi P$.
 \end{defn}
 
 An example of an object of $\cMF_\phi(X)$ is given by \[
%\xymatrixrowsep{10pt}\xymatrixcolsep{10pt}
\xymatrix{%begin xy matrix
G_\phi P:=\left(
	P\oplus \phi P,\mat{0&\xi_P\\ \eta_P&0}
	\right):& P \ar@/_1pc/[r]^{\eta_P} & 
	\phi P\ar@/_1pc/[l]_{\xi_P} 
	}%end xy matrix
\]

We will show that $\cMF_\phi(X)$ is a Frobenius category with projective-injective objects given by $G_\phi P$. We use the key observation that $\eta_P,\xi_P$ are not zero divisors. This implies that, for any object $(P,d)$ in $\cMF_\phi(X)$, the factorization
\[
%\xymatrixrowsep{10pt}\xymatrixcolsep{10pt}
\xymatrix{%begin xy matrix
P \ar[r]^{\eta_P}\ar@/_1pc/[rr]_{d} & 
	\phi P\ar[r]^{\th_P} &P
	}%end xy matrix
%	V\xrarrow{\eta_V} \phi V\xrarrow{\th_V}V
\]
is unique. Similarly, $d\th_P=\xi_P:\phi P\to P$ since $d\th_P\eta_P=d^2=\cdot t=\xi_P\eta_P$. 

Furthermore, for all morphisms $f:(V,d)\to (W,d)$, the following commutes.
\begin{equation}\label{eq: xi eta theta diagram}
%\xymatrixrowsep{10pt}\xymatrixcolsep{10pt}
\xymatrix{%begin xy matrix
\phi V\ar[r]^{\xi_V}\ar[d]_{\phi f}&V\ar[d]_f \ar[r]^{\eta_V} & 
	\phi V\ar[r]^{\th_V}\ar[d]^{\phi f}&V\ar[d]^f\\
\phi W\ar[r]^{\xi_W}&W \ar[r]^{\eta_W} & 
	\phi W\ar[r]^{\th_W} &W
	}%end xy matrix
%	V\xrarrow{\eta_V} \phi V\xrarrow{\th_V}V
\end{equation}
This is because $\eta_V$ is not a zero divisor:
$
	(f\th_V-\th_W \phi f)\eta_V=fd-\th_W \eta_W f=fd-df=0
$
implies $f\th_V=\th_W\phi f$. Note that the left two squares in the diagram commute for any morphism $f:V\to W$ in $\cP(X)=add\,RX$ since $\xi,\eta$ are natural transformations.

\begin{lem}
$\cMF_\phi(X)$ is an exact category where a sequence
\[
	(A,d)\cofib (B,d)\onto (C,d)
\] 
in $\cMF_\phi(X)$ is defined to be exact if the underlying sequence $A\to B\to C$ is split exact in $\cP(X)$.
\end{lem}

The proof is straighrforward and identical to the case of $X=S^1$ detailed in \cite{IT09}.

\begin{lem}
The functor $G_\phi:\cP(X)\to \cMF_\phi(X)$ is left adjoint to the forgetful functor $\cMF_\phi(X)\to\cP(X)$ and the functor $G_\phi\circ \phi^{-1}$ is right adjoint to the forgetful functor. In other words, we have natural isomorphisms:
\[
	\cMF_\phi(G_\phi V,(W,d))\cong \cP(V,W)\cong \cMF_\phi((V,d),G_\phi\phi^{-1}W)
\]
\end{lem}

\begin{proof}
A morphism $f:V\to W$ corresponds to the morphisms $(f,\th_W\phi f):G_\phi V\to (W,d)$ and $(\phi^{-1}(f\th_V),f):(V,d)\to G_\phi \phi^{-1}W$. This is verified by a calculation in which we should remember that $\th_W \phi f\neq f\th_V$. However, since $d$ commutes with itself, we have $\th_V\phi d=d\th_V$.

For example, to prove right adjunction, we need to verify that $\phi^{-1}(f\th_V)d=(\xi_{\phi^{-1}W})f$ and $\eta_{\phi^{-1}W}\phi^{-1}(f\th_V)=fd$. This follows from:
\[
	f\th_V\phi d=fd \th_V=f\xi_V=\xi_W\phi f
\]
Applying $\phi^{-1}$ to both sides gives $\phi^{-1}(f\th_V)d=\phi^{-1}(\xi_W)f=\xi_{\phi^{-1}W}f$ since $\xi$ is natural. Also
\[
	(\eta_W f\th_V)\eta_V=\eta_W fd=\phi f\eta_Vd=(\phi f\phi d)\eta_V
\]
Cancel $\eta_V$ from both sides and apply $\phi^{-1}$ to get $\eta_{\phi^{-1}W}\phi^{-1}(f\th_V)=fd$.

Verification of left adjunction is similar but easier.
\end{proof}

\begin{prop}\label{G Px is strongly indecomposable}
If $P=P_x$ is indecomposable in $\cP(X)$ then the endomorphism ring of $G_\phi P_x$ is isomorphic to the local ring $R[u]$ where $u^2=t$.
\end{prop}

\begin{proof}
Morphisms $G_\phi V\to (W,d)$ are given by $(f,\th_W\phi f)$ where $f:V\to W$ is any morphism in $\cP(X)$. If $V=P_x$ and $(W,d)=G_\phi P_x$ then $W=P_x\oplus \phi P_x$ and $f:P_x\to P_x\oplus \phi P_x$ is a sum of two morphisms $f_0=r_0f_{xx}:P_x\to P_x$ and $f_1=r_1f_{x\phi x }:P_x\to\phi P_x$ where $r_0,r_1\in R$. Since $\th_W=id_x\oplus\xi_x'$ we get the following formula for a general endomorphism of $G_\phi P_x$:
\[
	(f,\th_W\phi f)=\mat{f_0 & \xi_x'\phi f_1\\ f_1 & \phi f_0}=\mat{r_0 f_{xx}& tr_1 f_{\phi x x}\\ r_1f_{x\phi x }& r_0f_{\phi x \phi x }}
\]
Since the second column is determined by the first, we can write this as $(r_0,r_1)$. Then composition is given by $
	(r_0,r_1)(s_0,s_1)=(r_0s_0+tr_1s_1,r_1s_0+r_0s_1)
$ which is exactly the multiplication rule for $(r_0,r_1)=r_0+r_1u$ in $R[u]$ where $u^2=t$.
\end{proof}

\begin{cor}
Every object $G_\phi V$ can be expressed uniquely up to isomorphism as a direct sum of indecomposable objects $G_\phi P_{x_i}$ where $P_{x_i}$ are the components of $V$.
\end{cor}

\begin{proof}
If $V=\bigoplus P_{x_i}$ then $G_\phi V=\bigoplus G_\phi P_{x_i}$ where each $G_\phi P_{x_i}$ is indecomposable.
\end{proof}

\begin{thm}\label{thm: MF-phi(X) is Frobenius}
$\cMF_\phi(X)$ is a Frobenius category with projective-injective objects given by $G_\phi V$ for all $V$ in $\cP(X)$.
\end{thm}

\begin{proof}
By the lemma above, $G_\phi V$ is projective for all $V$ and $G_\phi \phi^{-1}W$ is injective for all $W$. To see this note that $\cMF_\phi(G_\phi V,-)=\cP(V,F(-))$ and $\cMF_\phi(-,G_\phi \phi^{-1}W)=\cP(F(-),W)$ are exact since the forgetful functor $F:\cMF_\phi(X)\to\cP(X)$ takes exact sequences to split exact sequences. Letting $W=\phi V$ we see that $G_\phi V$ is projective and injective for all $V$.

For every $(V,d)$ there is a quotient map $G_\phi V\to (V,d)$ adjoint to the identity map on $V$. Similarly there is a cofibration $(V,d)\to G_\phi \phi^{-1}V$. Thus, there are enough projective-injective objects. If $(V,d)$ is projective or injective these morphisms must split, making $(V,d)$ a summand of either $G_\phi V$ or $G_\phi \phi^{-1}V$. By the corollary, this implies $(V,d)$ is a direct sum of $G_\phi P_{x_i}$. Since $\cMF_\phi(X)$ is an exact category with enough projectives all of which are also injective, it is a Frobenius category by definition.
\end{proof}

%%%%%%%%%%%%%%%%%%%%%%%%%%%%%

%-----------------------------------------------------------------------------------------
%  	A4:  Twisted version, part 2
%-----------------------------------------------------------------------------------------

%\newpage

%-----------------------------------------------------------------------------------------
% 		EXAMPLES
%-----------------------------------------------------------------------------------------

\section{Cluster categories} %We construct several cluster categories as examples of the general construction $\cMF(X)$. In most of these examples, $X$ will be a cyclically ordered set. Other examples will be constructed in other papers \cite{IT11D,IT11G}.

Recall that the \emph{stable category} $\ul\cF$ of a Frobenius category $\cF$ has the same set of objects as $\cF$ with morphism sets:
\[
	\ul\cF(A,B)=\frac{\cF(A,B)}{A\to P\to B, P\text{ proj-inj}}
\]

\begin{thm}[Happel \cite{HappelBook}]
The stable category of a Frobenius category is triangulated.
\end{thm}

We denote the stable categories of $\cMF(X)$ and $\cMF_\phi(X)$ by $\cC(X)=\ul\cMF(X)$ and $\cC_\phi(X)=\ul\cMF_\phi(X)$. These triangulated categories will have a cluster structure if $X$ and $\phi$ are carefully chosen. \emph{Cluster categories} were first constructed by Buan-Marsh-Reineke-Reiten-Todorov (0402054) as orbit categories. This construction is an alternate construction in type $A$.

%-----------------------------------------------------------------------------------------
%  	B0:  Triangulated category, Krull-Schmidt
%-----------------------------------------------------------------------------------------

\subsection{Krull-Schmidt theorem}

If $X$ is cyclically ordered we will show that $\cMF(X)$ is Krull-Schmidt $R$-category with indecomposable objects $E(x,y)$ defined below with $x,y\in X$. This will imply that $\cC(X)$ is a Krull-Schmidt $\kk$-category with indecomposable objects $E(x,y)$ where $y\not\approx x$. The following easy proposition implies that $\cMF_\phi(X)$ and thus $\cC_\phi(X)$ will also be Krull-Schmidt categories.

\begin{prop}\label{prop: MF-phi(X) is closed under direct summands}
The subcategory $\cMF_\phi(X)$ is closed under direct summands in $\cMF(X)$.
\end{prop}

\begin{proof}
Suppose that $(P_1,d_1)\oplus (P_2,d_2)$ lies in $\cMF_\phi(X)$. Then $d_1\oplus d_2:P_1\oplus P_2\to P_1\oplus P_2$ factors through $\eta_1\oplus \eta_2:P_1\oplus P_2\to \phi P_1\oplus \phi P_2$. Then $d_1$ factors through $\eta_1$ and $d_2$ factors through $\eta_2$ making $(P_1,d_1)$ and $(P_2,d_2)$ objects of $\cMF_\phi(X)$.
\end{proof}

%-----------------------------------------------------------------------------------------
%  	B0:  Krull-Schmidt, part 2
%-----------------------------------------------------------------------------------------

%We will show that in certain cases, the stable category of $\cMF_\phi(X)$ is a cluster category. Most of these cases assume that $(X,c)$ is cyclically ordered which is equivalent to saying that the associated covering poset $\tilde X$ is totally ordered. In that case, we have a Krull-Schmidt theorem identifying all of the indecomposable objects of $\cMF_\phi(X)$. 

%\subsubsection{\ul{Krull-Schmidt theorem}}

To construct the objects $E(x,y)$, we first recall that $c(xyx)=c(yxy)=0$ if and only if $x\approx y$ if and only if $P_x\cong P_y$. 

\begin{defn} We define a sequence of three not necessarilty distinct elements $(x,y,z)$ in $X$ to be in \emph{cyclic order} if there exist liftings $\tilde x,\tilde y,\tilde z\in \tilde X$ (possibly $\tilde x\neq \tilde z$ when $x=z$) so that $\tilde x\le \tilde y\le \tilde z\le \s\tilde x$. For example, $(x,\phi x,\phi^2x)$ is in cyclic order for any admissible automorphism $\phi$ of $X$. 
\end{defn}

It is easy to see that, if $(xyz), (zwx)$ are in cyclic order and $x\not\approx z$ then $(yzw)$ is also in cyclic order since there is a lifting $\tilde w\in \tilde X$ of $w$ so that $\tilde x<\tilde z\le \tilde w\le \s^n\tilde x<\s\tilde z<\s^2\tilde x$ forcing $n=1$. We call this \emph{composition of cyclic order}. When the entire set $X$ is cyclically ordered, $(xyz)$ is in cyclic order if and only if either $c(xyz)=0$ or $x\approx z$.

\begin{defn}
Let $(X,c)$ be a cyclic poset and let $x,y\in X$ so that $c(xyx)=c(yxy)=1$. Then we define $E(x,y)$ to be the object in $\cMF(X)$ given by
\[
	E(x,y):=\left(
	P_x\oplus P_y,\mat{ 0 & f_{yx}\\ f_{xy} &0}
	\right).
\]
If $x\approx y$ then we define $E(x,y)$ and $E(x,y)'$ by
\[
	E(x,y):=\left(
	P_x\oplus P_y,\mat{ 0 & tf_{yx}\\ f_{xy} &0}
	\right),\qquad 	E(x,y)':=\left(
	P_x\oplus P_y,\mat{ 0 & f_{yx}\\ tf_{xy} &0}
	\right)
\]
%where $n=1-c(xyx)$. Thus $n=1$ if $P_x\cong P_y$ and, otherwise, $n=0$. 
For example, $E(x,\phi x)=\left(
	P_x\oplus P_{\phi x},\mat{ 0 & \xi_x\\ \eta_x &0}
	\right)=G_\phi P_x$ in $\cMF_\phi(X)$ is the projective-injective object  in $E(x,\phi x)$.
\end{defn}

We observe that $E(x,y)'\cong E(y,x)$ by the isomorphism which switches $P_x,P_y$. This second copy of the same object is convenient for notational symmetry and is used in the proof of Corollary \ref{cor: (X,X0) is a Frobenius cyclic poset} below. (In the notation of \cite{IT09},\cite{IT10}, $E(x,x)'=E(x,x+2\pi)$.)

%We note that, if $X$ is cyclically ordered, then $c(xyx)\le 1$ for all $x,y\in X$. Therefore, the indecomposable objects of 

\begin{lem}\label{lem: E(x,y) = E(y,x)}
$E(x,y)$ is isomorphic to $E(a,b)$ in $\cMF(X)$ if and only if they are isomorphic in $\cP(X)$. In particular, $E(x,y)\cong E(y,x)$.
\end{lem}

\begin{proof}
We only need to prove sufficiency. Suppose that $P_x\oplus P_y\cong P_a\oplus P_b$. Then either $x\approx a$ and $y\approx b$ in $X$ or $x\approx b$ and $y\approx a$. In the first case, it is clear that $E(x,y)\cong E(a,b)$. In the second case we have $x\not\approx y$ (otherwise we are in Case 1). Then the transposition isomorphism $P_x\oplus P_y\cong P_y\oplus P_x$ commutes with the operator $d$ and therefore gives an isomorphism $E(x,y)\cong E(y,x)$. But $E(y,x)\cong E(a,b)$ as in Case 1. So, $E(x,y)\cong E(a,b)$ in both cases.
\end{proof}

\begin{lem}\label{E(x,y) is strongly indecomposable}
The endomorphism ring of $E(x,y)$ is isomorphic to the local ring $R[\sqrt t]$.
\end{lem}

\begin{proof} This is true for $y=\phi x$ by Proposition \ref{G Px is strongly indecomposable} and the same proof works for any $x,y$. Alternatively, note that the condition $c(xyx)\le 1$ implies that the cyclic subposet $\{x,y\}$ of $X$ is cyclically ordered and therefore can be embedded in the circle $S^1$. Then the endomorphism ring of $E(x,y)$ is isomorphic to the endomorphism ring of the corresponding object of $\cMF(S^1)$ which was shown to have endomorphism ring $R[\sqrt t]$ in \cite{IT09}.
\end{proof}

\begin{lem}\label{E(x,y) is in MF(X) iff fx,y,finv x is in cyclic order}
The object $E(x,y)$ of $\cMF(X)$ lies in $\cMF_\phi(X)$ if and only if $(\phi x,y,\phi^{-1}x)$ is in cyclic order.
\end{lem}

\begin{proof}
Suppose that $E(x,y)$ lies in $\cMF_\phi(X)$. Then the condition that $f_{xy}:P_x\to P_y$ factors through $\eta:P_x\to P_{\phi x}$ is equivalent to the condition that $c(x,\phi x,y)=0$. The other condition, that $t^nf_{yx}:P_y\to P_x$ factors through $\eta:P_y\to P_{\phi y}$ is equivalent to either $x\approx y$ or $x\not\approx y$ and $c(y,\phi^{-1}x,x)=0$. In the second case we conclude that $(\phi x,y,\phi^{-1}x)$ is in cyclic order. In the first case we must have $x\approx \phi x\approx y$ which also implies that $(\phi x,y,\phi^{-1}x)$ is in cyclic order.

Conversely, suppose that $(\phi x,y,\phi^{-1}x)$ is in cyclic order. Then, either $\phi x\approx \phi^{-1}x$, which implies $x\approx \phi x$ making $c(x,\phi x,y)=0=c(y,\phi^{-1}x,x)$, so that $E(x,y)$ lies in $\cMF_\phi(X)$ or $\phi x\not\approx \phi^{-1}x$, which implies that $x\not\approx y$ and, since $(\phi^{-1}x,x,\phi x)$ is in cyclic order, it also implies that $(x,\phi x,y)$ and $(y,\phi^{-1}x,x)$ are in cyclic order by {composition of cyclic order}.
\end{proof}

\begin{thm}\label{Krull-Schmidt theorem}
Suppose that $(X,c)$ is cyclically order and $\phi$ is an admissible automorphism of $X$. Then every object of $\cMF_\phi(X)$ is isomorphic to a finite direct sum of objects of the form $E(x,y)$ where $(\phi x,y,\phi^{-1}x)$ is in cyclic order. Furthermore, each $E(x,y)$ is indecomposable and the indecomposable direct summands of any object of $\cMF_\phi(X)$ are uniquely determined up to isomorphism. Finally, $E(x,y)$ is projective injective in $\cMF_\phi(X)$ if and only if either $y\approx \phi x$ or $x\approx \phi y$.
\end{thm}

\begin{proof}
We use the fact that any finite cyclically order set $X_0$ has a cyclic order preserving embedding $X_0\into S^1$ and the fact that the theorem holds for $X=S^1$ and $\phi=id_{S^1}$ by \cite{IT09}. If $(P,d)$ is any object of $\cMF_\phi(X)$ then $P=\bigoplus P_{x_i}$. Let $X_0=\{x_i\}$. Since this is a finite cyclically ordered set, we conclude that $(P,d)$, as an object of $\cMF(X)$, decomposes into objects $E(x,y)$ with $x,y\in X_0$. By Proposition \ref{prop: MF-phi(X) is closed under direct summands}, the components of $(P,d)$ also lie in $\cMF_\phi(X)$. So, $(\phi x,y\phi^{-1}x)$ is in cyclic order by Lemma \ref{E(x,y) is in MF(X) iff fx,y,finv x is in cyclic order} above.

The uniqueness of decomposition follows from Lemma \ref{E(x,y) is strongly indecomposable}.

By Theorem \ref{thm: MF-phi(X) is Frobenius}, the indecomposable projective-injective objects of $\cMF_\phi(X)$ are $G_\phi P_z= E(z,\phi z)$. By Lemma \ref{lem: E(x,y) = E(y,x)}, the object $E(x,y)$ is isomorphic to such an object if and only if either $y\approx \phi x$ or $x\approx \phi y$.
\end{proof}

%We will now go over several examples.

\subsection{Frobenius cyclic posets}

We need the following corollaries and definitions as the starting points of our next two papers (\cite{IT11D},\cite{IT11G}).

\begin{defn}\label{defn: Frobenius cyclic poset of Z}
Suppose that $Z$ is a cyclically ordered set with totally ordered covering poset $\tilde Z$ and $\phi$ is an admissible automorphism of $Z$ with corresponding covering automorphism $\tilde\phi:\tilde Z\to\tilde Z$ so that $a\le \tilde\phi a<\s\tilde\phi^{-1}a\le\s a$ for all $a\in\tilde Z$. Then we define $\cX(Z,\phi)$ to be the cyclic poset given as follows. The covering poset of $\cX(Z,\phi)$ is the set
\[
	\tilde \cX=\{(a,b)\in \tilde Z\times\tilde Z\,|\, \tilde\phi a\le b\le\s\tilde\phi^{-1}a\}
\]
with partial ordering $(a,b)\le (a',b')$ if and only if $a\le a'$ and $b\le b'$ and
with automorphism $\s$ defined by $\s(a,b)=(b,\s a)$. Let $\cX_0(Z,\phi)$ be the cyclic subposet of $\cX(Z,\phi)$ given by the covering subposet
\[
	\tilde \cX_0=\{(a,b)\,|\, b\approx \phi a\text{ or }b\approx \phi^{-1}a\}.
\]
%We call the pair $(\cX,\cX_0)$ the \emph{associated Forbenius cyclic poset} of $(Z,\phi)$.
\end{defn}

\begin{cor}\label{cor: (X,X0) is a Frobenius cyclic poset} Let $R=\kk[[t]]$. Then the additive $R$-category $\cP(\cX(Z,\phi))=add\,R\cX$ has the structure of a Frobenius category with $\cP(\cX_0(Z,\phi))$ being the full subcategory of projective-injective objects.
\end{cor}

In other words, the pair $(\cX(Z,\phi),\cX_0(Z,\phi))$ satisfies the following definition. 

\begin{defn} A \emph{Frobenius cyclic poset} is defined to be a pair $(X,X_0)$ where $X$ is a cyclic poset and $X_0$ is a subset of $X$ with the induced cyclic poset structure such that, for $R=\kk[[t]]$ with $\kk$ any field, $\cP(X)=add\,RX$ has the structure of a Frobenius category with projective injective objects forming the full subcategory $\cP(X_0)$.
\end{defn}

We will see in \cite{IT11D} that the exact structure of such a Frobenius category is uniquely determined. So the words ``has the structure of a Frobenius category'' in the above definition can be replaced with ``has a uniquely determined structure of a Frobenius category''.

In order to prove the corollary it will be useful to ``double'' the cyclic poset $\cX$ since there are two objects in the Frobenius category $\cMF_\phi(Z,R_0)$ (where $R_0=\kk[[t^2]]$) for every object in $R\cX$. 

\begin{defn}\label{defn: doubled Frobenius cyclic poset of Z} For $Z,\tilde Z,\phi,\tilde\phi$ as in Definition \ref{defn: Frobenius cyclic poset of Z}, let $\cX\double=\cX\double(Z)$ denote the cyclic poset given by the covering poset $\tilde\cX\double$ which consists of two copies of $\tilde\cX$:
\[
	\tilde\cX\double=\tilde\cX\times \{+,-\}=\{(a,b)_\e\,|\, a,b\in \tilde Z, \e=\pm, \tilde\phi a\le b\le\s\tilde\phi^{-1}a\}
\]
%with elements $(a,b)_\e$ for $\e=\pm$ 
with partial ordering disregarding the sign: $(a,b)_\e\le (a',b')_{\e'}$ if $a\le a'$ and $b\le b'$. In particular, $(a,b)_+\approx (a,b)_-$. The automorphism $\s$ is given by $\s(a,b)_\e=(b,\s a)_{-\e}$. Let $\cX\double$ be the set of $\s$-orbits in $\tilde\cX\double$ and let $\cX_0\double$ be the subset of $\s$-orbits of $\tilde\cX_0\double=\tilde\cX_0\times\{\pm\}$.
\end{defn}

Clearly, $R\cX\double$ is equivalent to $R\cX$ and this equivalence sends $R\cX_0\double$ to $R\cX_0$.

\begin{proof}[Proof of Corollary \ref{cor: (X,X0) is a Frobenius cyclic poset}]
Let $R_0=\kk[[t^2]]$. This is a discrete valuation ring with uniformizer $v=t^2$ and $R=R_0[\sqrt{v}]$. Let $\cMF_\phi(Z,R_0)$ be the Frobenius category constructed from the cyclic poset $Z$ using $R_0,v$ instead of $R,t$ in Definitions \ref{defn: MF(X)}, \ref{defn: twisted MF-phi(X)}. Let $\cM(Z)$ be the full subcategory of $\cMF_\phi(Z,R_0)$ with objects $E(x,y)$ (where $(\phi x,y,\phi^{-1}x)$ are in cyclic order) and $E(x,y)'$ (where $x\approx \phi x\approx y$) and let $\cM_0(Z)$ be the full subcategory of $\cM(Z)$ of the projective-injective objects $E(x,y)$ (where either $y\approx \phi x$ or $x\approx \phi y$) and $E(x,y)'$ (where $x\approx \phi x\approx y$).

We claim that $\cM(Z)$ is isomorphic to $R\cX\double$ as $R$-categories and that this isomorphism sends $\cM_0(Z)$ to $R\cX_0\double$. This will imply that the additive categories that they generate are also equivalent and this equivalence will give the structure of a Frobenius category to $add\,R\cX\double$ with $add\,R\cX_0\double$ being the full subcategory of projective-injective objects, thereby proving the corollary.

A bijection between the objects of $\cM(Z)$ and the objects of $R\cX\double$ which are the elements of $\cX\double$ is given as follows. Let $\psi:\tilde\cX\double\to Ob\cM(Z)$ be defined by 
\[
	\psi(x,y)_+ = \begin{cases} E(\pi(x),\pi(y))' & \text{if }  y\approx \s\tilde\phi x\approx \s x\\
   E(\pi(x),\pi(y)) & \text{otherwise}
    \end{cases}
\]
and $\psi(x,y)_-=\psi(y,\s x)_+$. Then we claim that $\psi$ induces a bijection $\Psi$ between the set of $\s$-orbits in $\tilde\cX\double$ and the set of objects of $\cM(Z)$. The inverse of $\Psi$ is given as follows. If $x\not\approx y$ then $\Psi^{-1}E(x,y)$ is the $\s$-orbit of $(\tilde x,\tilde y)_+$ where $\tilde x,\tilde y\in\tilde Z$ are any lifting of $x,y$ so that $\tilde\phi\tilde x\le\tilde y\le\tilde\phi^{-1}\tilde x$. If $x\approx y$ then this is not well defined since there are two such $\s$ orbits given by $(\tilde x,\tilde y)_+$ and $(\tilde x,\s\tilde y)_+$ where $\tilde x\approx \tilde y$. In this case we let $\Psi^{-1}E(x,y)=(\tilde x,\tilde y)_+$ and $\Psi^{-1}E(x,y)'=(\tilde x,\s\tilde y)_+$. It is straightforward to show that these are inverse maps.

We will use the completed linearization $\widehat{\kk \cX\double}$ instead of $R\cX\double$ since they are isomorphic. (Definition \ref{defn: completed linearization} and Proposition \ref{prop: completed linearization is isomorphic to RX}.) We also use the fact that $\widehat{\kk Z}\cong R_0Z$. An isomorphism $\Psi:\widehat{\kk \cX\double}\to \cM(Z)$ is given on object above and on morphisms below.% as follows. On objects, $\Psi$ takes the $\s$-orbit of $(x,y)_+$ to $E(\pi(x),\pi(y))$ where $\pi:\tilde Z\to Z$ is the projection map.

Since morphisms on both sides are given by infinite series of monomials, we will describe the isomorphism only on monomial morphisms. There are two kinds of monomial morphisms: even and odd. An \emph{even} monomial morphism, call it $f$, from the $\s$ orbit of $(x,y)_+$ to that of $(a,b)_+$ is given by a scalar $r$ times inequalities $(\s^nx,\s^ny)_+\le (\s^{n+i}a,\s^{n+i}b)_+$ and $(\s^ny,\s^{n+1}x)_-\le (\s^{n+i}b,\s^{n+i+1}a)_-$ for all $n$ but fix $i$ and an \emph{odd} morphism, call it $g$, between the same two $\s$-orbits is given by a scalar $s$ times the inequalities $(\s^nx,\s^ny)_+\le (\s^{n+j}b,\s^{n+j+1}a)_-$ and $(\s^ny,\s^{n+1}x)_-\le (\s^{n+j+1}a,\s^{n+j+1}b)_+$ for all $n$ but fixed $j$.

The functor $\Psi$ takes the $\s$-orbit of $(x,y)_+$ to $
	\Psi(x,y)_+=(P_{\pi(x)}\oplus P_{\pi(y)},d)$ where $d$ is the counterdiagonal matrix whose entries are given, as morphisms in $\widehat{\kk Z}$, by the inequalities $x\le y$ and $y\le \s x$. The functor $\Psi$ takes the even morphism $f$ to the diagonal morphism $\Psi(x,y)_+\to \Psi(a,b)_+$ given by the $2\times 2$ matrix whose off-diagonal entries are zero and whose diagonal entries are $r$ times the morphism $\pi(x)\to \pi(a)$ in $\widehat{\kk Z}$ given by the inequalities $\s^nx\le\s^{n+i}a$ for all $n$ and $r$ times the morphism $\pi(y)\to \pi(b)$ given by $\s^ny\le \s^{n+i}b$ for all $n$. The functor $\Psi$ also takes the odd morphism $g$ to the matrix whose diagonal entries are zero and whose off-diagonal entries are $s$ times the inequalities $\s^nx\le\s^{n+j}b$ and $\s^ny\le \s^{n+j+1}a$. The functor $\Psi$ commutes with composition since composition is defined component-wise on both sourse and target. I.e., $\Psi$ lies over the isomorphism $add\,\widehat{\kk Z}\to add\,R_0Z$. Therefore, $\Psi$ gives an isomorphism of categories $\Psi:\widehat{\kk \cX\double}\cong \cM(Z)$.
\end{proof}

%An isomorphism $\Psi:\cM(Z)\to R\cX\double$ is given as follows. For any object $E(x,y)$ in $\cM(Z)$, let $\tilde x\in \tilde Z$ be any lifting of $x$ and let $\tilde y\in\tilde Z$ be the unique lifting of $y$ with the property that $\s\tilde\phi^{-1}\tilde x\ge\tilde y\ge\tilde\phi\tilde x>\s^{-1}\tilde y$. Since $\phi x,y,\phi^{-1}x$ are in cyclic order, such a $\tilde y$ exists. Let $\Psi E(x,y)$ be the $\s$-orbit of $(\tilde x,\tilde y)_+$ which we denote $[\tilde x,\tilde y]_+,$. This is independent of the choice of $\tilde x$ since $(\s\tilde x,\s\tilde y)_\e=\s^2(\tilde x,\tilde y)_\e$.

%\newpage

%-----------------------------------------------------------------------------------------
%  	B1:  Continuous cluster categories
%-----------------------------------------------------------------------------------------

\subsection{Continuous cluster categories} In this family of examples, $X$ is the cyclically ordered set $X=S^1=\RR/2\pi\ZZ$ with covering poset $\tilde X=\RR$ with $\s(x)=x+2\pi$. For any fixed $0\le \th<\pi$, let $\phi$ be the  automorphism of $S^1$ with lifting $\tilde \phi(x)=x+\th$. Then $\phi$ is admissible and the Frobenius category $\cMF_\phi(S^1)$ is equal to the continuous Frobenius category $\cF_{\pi-\th}$ with stable category $\cC_{\pi-\th}$. (See \cite{IT09,IT10}.) In subsequent papers \cite{IT11D,IT11G} we use the fact (Corollary \ref{cor: (X,X0) is a Frobenius cyclic poset}) that, for $R=\kk[[t]]$ and $R_0=\kk[[t^2]]$, the Frobenius category $\cMF_\phi(S^1,R_0)$ is isomorphic to the completed linearization of the Frobenius cyclic poset $(\cX,\cX_0)$ where
\[
	\tilde\cX=\{(x,y)_\e\in \RR^2\times\{\pm\}\,|\, x+\th\le y\le x+2\pi-\th\}
\]
partially ordered by $(x,y)_\e\le (x',y')_{\e'}$ if $x\le x'$ and $y\le y'$ and $\s(x,y)_\e=(y,x+2\pi)_{-\e}$.

The indecomposable objects of $\cC_{\pi-\th}$ are $E(x,y)$ where $x,y$ are distinct points on the circle subtending an angle more than $\th$. This triangulated category has a cluster structure if and only if $\th=2\pi/(n+3)$ for some positive integer $n$ (Cor. 5.4.4 in \cite{IT10}).

%\newpage
%-----------------------------------------------------------------------------------------
%  	B2:  Cluster category of type An
%-----------------------------------------------------------------------------------------

\subsection{Discrete cluster category of type $A$} The basic examples of discrete cluster categories of type $A$ are given by the cyclic posets $Z_n=\{1,2\cdots,n\}$ where $n\ge3$ with $\phi(i)=i+1$ modulo $m$ which give the cluster category of type $A_{n-3}$ and the cyclic poset $\ZZ$, with $\phi(i)=i+1$, which gives the $\infty$-gon of \cite{HJ12}. Both are examples of the following construction.

%By a \emph{discrete cluster category of type $A$} over a field $\kk$ we mean the stable category of a Frobenius cyclic poset $(\cX\double,\cX_0\double)$ constructed as follows. 

\begin{thm}\label{thm: C(Z) is a 2-CY cluster category}
Let $Z$ be a cyclically ordered set having at least four elements with totally ordered covering $\tilde Z$ and let $\phi$ be an admissible automorphism of $Z$ satisfying the following.
\begin{enumerate}
\item $x<\tilde\phi x<\tilde\phi^2 x<\s x$ for all $x\in\tilde Z$.% and
\item There are no elements $z\in\tilde Z$ so that $x<z<\tilde\phi x$ for any $x\in\tilde Z$. 
\end{enumerate}
Then the stable category of the Frobenius category $\cMF_\phi(Z)$ is a $2$-Calabi-Yau cluster category.
\end{thm}

This is a mild generalization of the results of \cite{J04}, \cite{HJ12}. We refer to \cite{HJ12} for many of the proofs. We will give the definitions and statements so that the reader can make the comparison and we give one new example: when $Z=\ZZ^2$ with lexicographic order.

%We will first show that $\cC_\phi(Z)$ is $2$-Calabi-Yau, then show that it has a cluster structure.

\subsubsection{$2$-Calabi-Yau property}

For convenience of notation we assume that $Z$ is infinite and that no two distinct elements of $Z$ are equivalent. Then we can take $Z$ to be a totally ordered set and $\phi$ to be an automorphism of $Z$ so that $\phi x$ is the smallest element of $Z$ which is greater than $x$ for every $x\in Z$. We call $\phi x$ the \emph{successor} of $x$ and denote it by $x^+$. Similarly, we call $\phi^{-1}x$ the \emph{predecessor} of $x$ and we denote it by $x^-$. %By choosing one element from each equivalence class, we may assume that no two distinct element of $Z$ or $\tilde Z$ are equivalent.

The indecomposable objects of $\cMF_\phi(Z)$ are $E(x,y)\cong E(y,x)$ where $x,y$ are distinct elements of $Z$. We denote the one with first coordinate small than the second by the set $X=\{x,y\}$. (Thus $X=E(x,y)$ if $x<y$.) Such a subset $X$ will denote a projective-injective object if and only if the larger element of $X$ is the successor of the smaller element.

Finally, we observe that the stable category of $\cMF_\phi(Z)$ depends only on $\kk=R/\mathfrak m$ and is independent of the choice of $R$ since multiplication by $t$ always factors through a projective-injective object. Therefore, we may assume that $R=\kk[[t]]$ and $\cMF_\phi(Z)=\cP(\cX)$ and $\cC_\phi(Z)=\cP(\cX)/\cP(\cX_0)$ with $\cX=\cX(Z),\cX_0=\cX_0(Z)$ given by Definition \ref{defn: Frobenius cyclic poset of Z}. Since $t=0$ in the stable category, $\cC_\phi(X,Y)=\kk$ or $0$ for all indecomposable objects $X,Y$.

\begin{lem}
If $X,Y$ are 2-element subsets of $Z$ which are not projective-injective then, in the stable category $\cC_\phi(Z)$ of $\cMF_\phi(Z)$, we have $\cC_\phi(Z)(X,Y)=\kk$ if
\begin{equation}\label{eq: conditions for nonzero morphism in C(Z)}
	x_0\le y_0<x_1^-,\quad x_1\le y_1<\s x_0^-
\end{equation}
for some liftings $x_i,y_i\in\tilde Z$ of the elements of $X,Y$ and $\cC_\phi(Z)(X,Y)=0$ otherwise. Furthermore, any nonzero morphism $X\to Y$ factors through a 2-element subset $S$ of $Z$ if and only if
\[
	x_0\le s_0\le y_0,\quad x_1\le s_1\le y_1
\]
for some liftings $s_i\in\tilde Z$ of the elements of $S$ (with the $x_i,y_i$ satisfying \eqref{eq: conditions for nonzero morphism in C(Z)} above).
\end{lem}

\begin{proof}Choose liftings $x_0,x_1\in\tilde Z$ for the elements of $X$ so that $x_0<x_1<\s x_0$. If $\pi^{-1}Y$ has no elements in the half open interval $[x_0,x_1^-)$ then any morphism $X\to Y$ in $\cMF_\phi(Z)$ will factor through the projective-injective object $\pi\{x_0,x_1^-\}$. So, in order to have a nonzero morphism $X\to Y$ in $\cC_\phi(Z)$, there must be a lifting $y_0\in[x_0,x_1^-)$ of one of the elements of $Y$. Similarly, there must be a lifting $y_1\in[x_1,\s x_0^-)$ of the other element of $Y$. Therefore, \eqref{eq: conditions for nonzero morphism in C(Z)} is necessary to have a nonzero morphism $X\to Y$.

The statement about when a morphism $X\to Y$ factors through $S$ is clear and, assuming \eqref{eq: conditions for nonzero morphism in C(Z)}, $S$ cannot be projective-injective. Therefore $\cC_\phi(Z)(X,Y)\neq0$ when \eqref{eq: conditions for nonzero morphism in C(Z)} hold.\end{proof}

\begin{lem} In the triangulated category $\cC_\phi(Z)$ we have $E(x_0,x_1)[1]=E(x_1^-,x_0^-)$. Furthermore, the shift functor $[1]$ takes basic even morphisms to basic even morphisms and basic odd morphisms to negative basic odd morphisms.
\end{lem}

\begin{proof}For any $x_0\neq x_1$ in $Z$ we choose the exact sequence in $\cMF_\phi(Z)$:
\[\xymatrixrowsep{0pt}\xymatrixcolsep{30pt}
\xymatrix{%begin xy matrix
&E(x_0,x_0^-)\ar[dr]^{1}\\
E(x_0,x_1)\ar[dr]_{1}\ar[ur]^{-1} &	\oplus &	E(x_1^-,x_0^-)\\
& 	E(x_1^-,x_1)\ar[ur]_{1}	}%end xy matrix
\]
where the middle term is the injective envelope of $X=E(x_0,x_1)$ and the projective cover of $E(x_1^-,x_0^-)$. All four morphisms are even morphisms which are plus or minus a basic morphism as indicated. With this choice of injective envelopes for all indecomposable objects of $\cMF_\phi(Z)$ we obtain $E(x_1^-,x_0^-)=E(x_0,x_1)[1]$ in the stable category $\cC_\phi(Z)$. It is clear that this is natural with respect to even morphisms, i.e., that $[1]$ takes basic even morphisms to basic even morphisms.

If we switch the order of $x_0,x_1$ then the signs on the two morphisms on the left will change. Therefore $\eta[1]=-\eta$ where $\eta:E(x_0,x_1)\cong E(x_1,x_0)$ is the basic odd isomorphism. This implies that $[1]$ changes the sign of all odd morphisms.\end{proof}

\begin{lem}
$\Ext^1(X,Y)\neq0$ if and only if the subsets $X,Y$ of $Z$ are ``crossing'' in the sense that
\[
	x_0<y_0<x_1<y_1<\s x_0
\]
for some liftings $x_i,y_i\in \tilde Z$ of the elements of $X,Y$. In particular, $\Ext^1(X,X)=0$.
\end{lem}

\begin{proof}This follows from the previous two lemmas.\end{proof}

\begin{thm}\label{thm: C(Z) is 2-CY}
$\cC_\phi(Z)$ is $2$-Calabi-Yau, i.e., there is a natural isomorphism 
\[
\Ext^1(X,Y)\cong D\Ext^1(Y,X)
\]
%and all indecomposable objects are rigid in the sense that $\Ext^1(X,X)=0$.
\end{thm}

\begin{proof}The lemma above implies that $Ext^1(X,Y)\cong DExt^1(Y,X)$ for any two indecomposable objects $X,Y$. It remains to show that the isomorphism is natural. Equivalently, we need to define a natural nondegenerate pairing
\[\<\cdot,\cdot\>:\Ext^1(X,Y)\otimes \Ext^1(Y,X)\to \kk\]
Since $\Ext^1(X,Y)=(X,Y[1])$ by definition, this pairing can be defined as the composition
\[(X,Y[1])\otimes (Y,X[1])\xrarrow{[1]\otimes id}(X[1],Y[2])\otimes (Y,X[1])\xrarrow{\circ}(Y,Y[2])\xrarrow{\Tr}\kk\]
where the trace map $\Tr:(Y,Y[2])\to\kk$ is given by choosing a decomposition of $Y$ into indecomposable objects $Y=\bigoplus Y_i$ and letting $\Tr(f)=\sum f_{ii}$ where $f_{ij}\in\kk$ is the scalar corresponding to the $ij$ component of $f$ (since $(Y_j,Y_i)=\kk$ or 0). Thus
\[	\<f,g\>=\Tr(f[1]\circ g)\]
It is clear that this pairing is natural in $X$. Naturality in $Y$ comes from the fact that $[2]$ takes basic morphisms to basic morphisms (without changing sign) and thus $\Tr(g\circ h)=\Tr(h[2]\circ g)$. Naturality in $Y$ also follows from the antisymmetry of the pairing:
\[\<f,g\>=-\<g,f\>\]
which follows from the fact that, for indecomposable $X$, any morphism $X\to X[2]$ must be odd. So, one of the morphisms $f$ or $g$ is odd and the other is even and $[1],[-1]$ will reverse the sign of the odd morphism and preserve the sign of the even morphism making $\Tr(f[1]\circ g)=-\Tr(g[1]\circ f)$.\end{proof}

\subsubsection{Clusters} Clusters in $\cC_\phi(Z)$ are maximal compatible sets of indecomposable objects which satisfy a certain continuity property which we now explain.

We recall that a \emph{Dedekind cut} in $\tilde Z$ is a nonempty proper subset $S$ of $\tilde Z$ so that if $x<y$ and $y\in S$ then $x\in S$. We consider only \emph{proper} Dedekind cuts, i.e., those which do not have suprema in $\tilde Z$. Note that $\s S\backslash S$ maps bijectively to $Z$ and induces a total ordering on the set $Z$ so that $Z$ does not have a maximum or minimum element. We call this a \emph{Dedekind ordering} on $Z$. We say that a family of elements $z_\a$ in $Z$ \emph{converges to $\infty$} with respect to such an ordering if for all $x\in Z$ there exists $\a$ so that $z_\a>x$. Convergence to $-\infty$ is defined similarly.

Note that any nonempty subset $T$ of $Z$ either has a supremum or converges to $\infty$ with respect to another Dedekind ordering of $Z$. (Take the Dedekind cut $S$ consisting of all $z\in Z$ which are less than some element of $T$.)
%It is an easy exercise in the use of Zorn's lemma to show that, for every Dedekind ordering of $Z$, there exists a strictly increasing transfinite sequence $z_\a$ in $Z$ converging to $\infty$ and a strictly decreasing transfinite sequence $z_\b$ converging to $-\infty$.

For indecomposable $X\cong E(x,y)$, the points $x,y$ are called the \emph{endpoints} of $X$.

\begin{defn}
A \emph{cluster} if $\cC_\phi(Z)$ is defined to be a maximal collection $\cT$ of nonisomorphic indecomposable objects $T_i$ satisfying two properties:
\begin{enumerate}
\item (Compatibility) $\Ext^1(T_i,T_j)=0$ for all $i,j$.
\item (Limit Condition) Suppose that $T_\a$ is a transfinite sequence of objects in $\cT$ so that one end $x_\a$ of $T_\a$ converges to $\infty$ with respect to some Dedekind ordering of $Z$ and the other end of $T_\a$ if fixed at, say $x$, then there is another transfinite sequence $S_\b$ of objects in $\cT$ so that one end of $S_\b$ is fixed at the same point $x$ and the other endpoint of $S_\b$ converges to $-\infty$ with respect to the same Dedekind ordering of $Z$.
\end{enumerate}
\end{defn}

We say that indecomposable $X,Y$ in $\cC_\phi(Z)$ are \emph{compatible} if $\Ext^1(X,Y)=0=\Ext^1(Y,X)$.

Note that Limit Condition (2) is automatically satisfied if $\cT$ is \emph{locally finite} in the sense that, for every $x\in Z$, there are only finitely many objects in $\cT$ with one end at $x$. In the case when $Z=\ZZ$, this definition is equivalent to the condition in \cite{HJ12} Theorem B and the proof that such sets satisfy the definition of a cluster is analogous. So, we omit the proof. However, we need to prove the symmetry of the above Limit Condition: 

\begin{lem}
The Limit Condition implies its converse, i.e., the same statement holds if $\infty$ and $-\infty$ are reversed.
\end{lem}

\begin{proof}Suppose that $T_\a$ is a collection of objects in $\cT$ with one end fixed at $x$ and the other end converging to $-\infty$.

\ul{Claim 1} $\cT$ contains an object isomorphic to $E(x,y)$ for some $y>x$.

If not then, in particular, $E(x,x^{++})$ is not in $\cT$ which is equivalent to saying that $\cT$ has an object isomorphic to $E(x^+,y)$ where $y>x^+$. But the set of all such $y$ must have a supremum. Otherwise, by the Limit Condition, the $y$'s must converge to $\infty$ which is impossible since $E(x^+,z)$ is not compatible with the objects $T_\a$ for any $z<x^+$. Let $y_0$ be this supremum. Then $E(x,y_0)$ is compatible with all objects in $\cT$ and therefore an object in $\cT$ (up to isomorphism). This proved Claim 1.

We want to show that the set of all $y>x$ so that $E(x,y)$ is in $\cT$ (up to isomorphism) converges to $\infty$. If this is not the case then this set must have a supremum since, by the Limit Condition, if it converges to some Dedekind cut from below then it also converges to it from above. So, let $y_1$ be the supremum of this set. Then consider the set of all $z>y_1$ so that $\cT$ contains an object isomorphic to $E(y_1,z)$. As in Claim 1, this set is nonempty and contains a maximal element $z_1$. But then $E(x,z_1)$ is compatible with all objects in $\cT$ but is not contained in $\cT$ since $z_1>y_1$. This contradicts the maximality of $\cT$ and this contradiction proved the lemma.\end{proof}

\begin{lem}
$\cC_\phi(Z)$ contains at least one cluster. Furthermore, for any object $T$ of any cluster $\cT$, there is, up to isomorphism, a unique object $T^\ast$ so that $\cT\backslash T\cup T^\ast$ is a cluster.
\end{lem}

As in \cite{HJ12}, the proof depends on the following.

\begin{lem}
For any object $T=E(x,y)$ in any cluster $\cT$ of $\cC_\phi(Z)$, there are unique elements $a,b\in Z$ so that $\tilde x<\tilde a<\tilde y<\tilde b<\s \tilde x$ for some liftings $\tilde x,\tilde y,\tilde a,\tilde b\in\tilde Z$ of $x,y,a,b$ and so that $E(x,a),E(a,y),E(y,b),E(b,x)$ are either zero or isomorphic to objects in $\cT$.\qed
\end{lem}

This implies that we have exact sequences $E(x,y)\to E(x,b)\oplus E(a,y)\to E(a,b)$ and $E(a,b)\to E(a,x)\oplus E(y,b)\to E(x,y)$ in the Frobenius category $\cMF_\phi(Z)$ giving the following distinguished triangles in the triangulated category $\cC_\phi(Z)$:
\[
	T=E(x,y)\to E(x,b)\oplus E(a,y)\to T^\ast=E(a,b)\to T[1]
\]
\[
	T[-1]\to T^\ast=E(a,b)\to E(a,x)\oplus E(y,b)\to T=E(x,y)
\]
where, up to isomorphism, $B=E(x,b)\oplus E(a,y)$ is a right $add(\cC_\phi(Z)\backslash T)$-approximation of $T$ and $B'=E(a,x)\oplus E(y,b)$ is a left $add(\cC_\phi(Z)\backslash T)$-approximation of $T$. As in \cite{HJ12} and as outlined below, this implies that $\cC_\phi(Z)$ is a cluster category in the sense that it has a cluster structure according to the following definition from \cite{BIRSc}.

\begin{defn}
Suppose that $\cC$ is a triangulated Krull-Schmidt category. Then a \emph{cluster structure} on $\cC$ is a collection of sets $\cT$ called \emph{clusters} of nonisomorphic indecomposable objects called \emph{variables} satisfying the following conditions.
\begin{enumerate}
\item[(a)] For any cluster variable $T$ in any cluster $\cT$ there is, up to isomorphism, a unique object $T^\ast$ not isomorphic to $T$ so that $\cT^\ast:=\cT\backslash T\cup T^\ast$ is a cluster.
\item[(b)] There are short exact sequences (or distringuished triangles) 
\[
	T^\ast\to B\to T,\qquad T\to B'\to T^\ast
\]
so that $B$ is a minimal right $\add(\cT\backslash T)$-approximation of $T$ and $B'$ is a minimal left $\add(\cT\backslash T)$-approximation of $T$. We write $B=B_\cT(T)$ and $B'=B_\cT'(T)$.
\item[(c)] There are no loops or 2-cycles in the quiver of any cluster $\cT$. (No loops means that any nonisomorphism $T\to T$ factors through $B_\cT(T)$ and no $2$-cycles means that there do not exist cluster variables $T,S$ in $\cT$ so that $S$ is a summand of $B_\cT(T)$ and $T$ is a summand of $B_\cT(S)$.
\item[(d)] The quiver of $\cT^\ast$ is obtained from the quiver of $\cT$ by Fomin-Zelevinski mutation. 
\item[(e)] If $\cT'$ is obtained from $\cT$ by replacing each variable with an isomorphic object then $\cT'$ is a cluster.
\end{enumerate}
\end{defn}

\begin{proof}[Proof of Theorem \ref{thm: C(Z) is a 2-CY cluster category}]
We have already shown conditions (a) and (b) and condition (e) holds by definition. Condition (c) is easy since the only arrows in the quivers of $\cT$  starting or ending at $T=E(x,y)$ and the only arrow in the quiver of $\cT\backslash T\cup T^\ast$ starting or ending at $T^\ast=E(a,b)$ are given in the following diagram.
\[\xymatrixrowsep{10pt}\xymatrixcolsep{10pt}
\xymatrix{%begin xy matrix
E(x,b)\ar[rr]  & &E(b,y)\ar[dl] && E(x,b)\ar[dr]  & &E(b,y)\ar[dd]\\
& E(x,y)\ar[ul]\ar[dr] &&\leftrightarrow && E(a,b)\ar[dl]\ar[ur]\\
E(x,a)\ar[ru]&&E(a,y)\ar[ll]&&E(x,a)\ar[uu]&&E(a,y)\ar[ul]	}%end xy matrix
\]
Some of the terms may be zero and should be deleted. Condition (d) holds by examination of the above diagram. Therefore, $\cC_\phi(Z)$ has a cluster structure. By Theorem \ref{thm: C(Z) is 2-CY}, it is $2$-Calabi-Yau.\end{proof}

\subsubsection{Example} We discuss explicit examples of discrete cluster categories.

\begin{lem}\label{almost split triangles}
$\cC_\phi(Z)$ has almost split triangles given by
\[
	E(x,y)\to E(x,y^+)\oplus E(x^+,y)\to E(x^+,y^+)\to E(y^-,x^-)
\]
where $y\neq x^{++}$ and
\[
	E(x^-,x^+)\to E(x^-,x^{++})\to E(x,x^{++})\to E(x, x^{--})
\]
Up to isomorphism, the irreducible maps are the basic morphisms $E(x,y)\to E(x,y^+)$.
\end{lem}

\begin{proof} The fact that the first sequence is a distinguished triangle (up to the signs of the morphisms) follows from that fact that $E(x,y)\to E(x,y^+)\oplus E(x^+,y)\to E(x^+,y^+)$ is an exact sequence in the Frobenius category. The second sequence is an example of a ``positive triangle'' \cite{IT10}. This is a distinguished triangle where all three morphisms are basic.

It is easy to see that these are almost split triangles: Given any morphism $f:E(a,b)\to E(x^+,y^+)$ which is not an isomorphism, assuming that $f$ is diagonal, either $a\neq x^+$ in which case $f$ factors through $E(x,y^+)$ or $b\neq y$ in which case $f$ factors through $E(x^+,y)$. Similarly, any nonisomorphism $E(x,y)\to E(a,b)$ factors through $E(x,y^+)\oplus E(x^+,y)$. So, the first triangle is almost split. A similar argument shows that the second triangle is almost split. The description of the irreducible morphisms is clear since $x\to x^+$ are the irreducible morphisms in $\cP(Z)$.
\end{proof}

\begin{lem}
Let $(Z,\phi)$ be as in Theorem \ref{thm: C(Z) is a 2-CY cluster category}. Then either $Z\cong Z_n$ with $\phi(i)=i+1$ or $Z\cong S\ast\ZZ$ where $S$ is cyclically ordered and $\phi(s,i)=(s,i+1)$.
\end{lem}

\begin{proof}
We prove only the infinite case. The automorphism $\phi$ gives an action of the additive group $(\ZZ,+)$ on the set $Z$. When $Z$ is infinite, this must be a free action. Let $S$ be the set of orbits of this action. Then $S$ is cyclically ordered and $Z\cong S\ast\ZZ$. 
\end{proof}

\begin{thm} The Auslander-Reiten quiver of $\cC_\phi(Z)\cong \cC_\phi(S\ast\ZZ)$ is a union of $\ZZ A_\infty$ components $C_s$ indexed by the elements of $s\in S$ and $\ZZ A_\infty^\infty$ components $C_{ab}$ indexed by unordered pairs of distinct elements of $S$.
\end{thm}

\begin{proof}
For every $s\in S$, the $C_s$ component consists of all objects $E((s,i),(s,j))$ where $i,j\in \ZZ$ with $|i-j|\ge2$. Since these contain the simple objects $E((s,i),(s,i+2))$, they are of type $\ZZ A_\infty$. For $a\neq b\in S$, the component $C_{ab}$ contains all objects $E((a,i),(b,j))$ where $i,j\in\ZZ$. This is a component of type $\ZZ A_\infty^\infty$.
\end{proof}

\begin{center}
\begin{figure}[ht]\label{zig-zag figure}
       \includegraphics[width=5in]{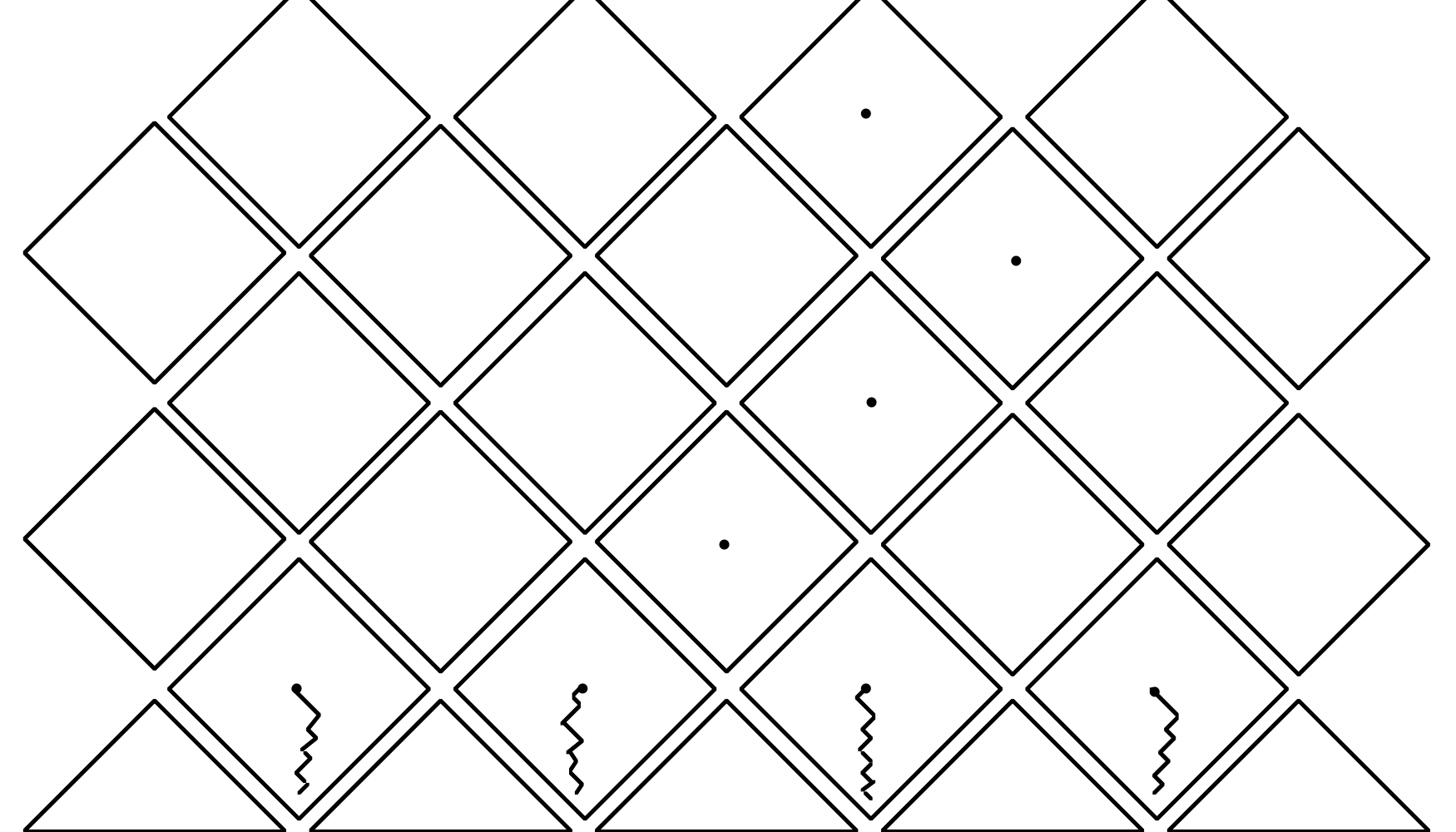}
\caption{An example of a cluster in $\cC_\phi(\ZZ^2)$. In each $C_{i,i+1}$ component, choose a zig-zag starting at the center point ($(0,0)$ followed by an infinite sequence of objects given by adding $(1,0)$ or $(0,-1)$ at each step and converging to $(\infty,-\infty)$). Then take the center points of $C_{ij}$ where $(i,j)$ form another zig-zag patter, such as $(0,2),(0,3),(0,4),(-1,4),\cdots$ (adding $(0,1)$ or $(-1,0)$ at each step and converging to $(-\infty,\infty)$). }
\end{figure}
\end{center}

An example of a cluster in $\cC_\phi(Z)$ for $S=\ZZ$ ($Z=\ZZ^2$ with lexicographic order) is given by taking a ``zig-zag'' in each component $C_{i,i+1}$ starting at the center point $E((i,0),(i+1,0))$ plus a zig-zag pattern of center points. By a \emph{zig-zag} in $C_{ab}$ we mean a sequence of objects $X_0,X_1,\cdots$ so that, if $X_k=E((a,i),(b,j))$ then $X_{k+1}=E((a,i+1),(b,j))$ or $E((a,i),(b,j-1))$ and so that $i\to\infty$ and $j\to-\infty$ as $k\to \infty$. By a \emph{zig-zag pattern of center points} we mean a sequence of objects $Y_0,Y_1,\cdots$ where each $Y_k=E((i,0),(j,0))$ and $Y_{k+1}=E((i-1,0),(j,0))$ or $E((i,0),(j+1,0))$ and so that $Y_0=E((i,0),(i+1,0))$ for some $i$ and so that $i\to-\infty$ and $j\to\infty$ as $k\to\infty$. See Figure 1.

%\newpage
%-----------------------------------------------------------------------------------------
%  	B4:  Infinite m-cluster categories
%-----------------------------------------------------------------------------------------

\subsection{$m$-cluster category of type $A_\infty$}

As another example, we will construct a triangulated category $\cC_\Phi(Z_m\ast\ZZ)$ whose \emph{standard objects} form a thick subcategory $\cC_\infty^m$ which satisfies the definition of an $m$-cluster category for $m\ge3$. We will show that the \emph{standard $m$-clusters} are in bijection with the set of partitions of the $\infty$-gon into $m+2$-gons. Such a category has already been constructed in \cite{HJ12b}. We also show that the triangulated category $\cC_\Phi(Z_m\ast\ZZ)$ is $m+1$-Calabi-Yau and that its $m+1$-rigid objects have a partial cluster structure. We call maximal compatible sets of $m+1$-rigid objects \emph{nonstandard clusters}. These correspond to 2-periodic partitions of the double $\infty$-gon into $m+2$-gons plus one $m+3$-gon in the middle. The sides of this middle polygon have more than $m$ mutations. These structures can also be obtained by taking $m$-cluster categories of type $A_n$ as in \cite{BM1} and taking the limit as $n$ goes to $\infty$.

%-----------------------------------------------------------------------------------
%            sub section Definitions
%-----------------------------------------------------------------------------------

\subsubsection{Definitions}

Let $m\ge3$ and let $Z_\infty^m=Z_m\ast\ZZ$ with elements $(p,k)$ denoted by $x_k^p$ where $p=1,\cdots,m,k\in\ZZ$. Let $\Phi:Z_\infty^m\to Z_\infty^m$ be the automorphism given by
\[
	\Phi(x^p_k)=\begin{cases} x^{p+1}_k & \text{if }p\neq m\\
    x^1_{k+1}& \text{if }p=m
    \end{cases}
\]
Then $\Phi^m(x_k^p)=x_{k+1}^p$ which is the successor of $x_k^p$ for any $x_k^p\in Z_\infty^m$. Since $x,\Phi(x),\Phi^2(x)$ are in cyclic order, $\Phi$ is an admissible automorphism of $Z_\infty^m$. To simplify notation we use the convention that $x_k^{p+m}=x_{k+1}^p$ for all integers $p,k$ and, more generally, $x_k^p=x_j^q$ if $mk+p=mj+q$. In other words, $\ll(x_k^p)=mk+p$ gives a bijection $\ll:Z_\infty^m\to \ZZ$. For example, the formula for $\Phi$ is $\Phi(x^p_k)= x^{p+1}_k$ for all integers $k,m$. Note that $\ll$ does not preserve order.

It is important to observe that $\Phi^m(x_k^p)=x_k^{p+m}=x_{k+1}^p$ is the \emph{successor} of $x_k^p$.

By Theorem \ref{Krull-Schmidt theorem} we have:

\begin{prop} $\cMF_\Phi(Z_\infty^m)$ is a Frobenius category with indecomposable objects $E(x,y)$ where $x,y\in Z_\infty^m$ so that $(\Phi(x),y,\Phi^{-1}(x))$ is in cyclic order, i.e., either $\Phi^{-1}(x)<\Phi(x)\le y$,  $y\le\Phi^{-1}(x)<\Phi(x)$ or $\Phi(x)\le y\le \Phi^{-1}(x)$. Also, we have:
\[
	E(x_i^p,x_j^q)[1]=E(\Phi^{-1}x_j^q,\Phi^{-1}x_i^p)=E(x_j^{q-1},x_i^{p-1})
\]
The shift operator $[1]$ takes basic even morphism to basic even morphisms and basic odd morphisms to $-1$ times basic odd morphisms.
\end{prop}

Since $\Phi^{-m}(x_i^p)=x_i^{p-m}=x_{i-1}^p$ is the predecessor of $x_i^p$, we conclude that \[
E(x,y)[m]=E(x^-,y^-)\]

Lemma \ref{almost split triangles} still holds, except that the formula for the shift $[1]$ has changed. The new lemma is:

\begin{lem}
The stable category $\ul\cMF_\Phi(Z_\infty^m)=\cC_\Phi(Z_\infty^m)$ has almost split triangles
\[
	\t E(x,y)\to E(x^-,y)\oplus E(x,y^-)\to E(x,y)\to \t E(x,y)[1]
\]
where $\t E(x,y)=E(x^-,y^-)\cong E(x,y)[m]$.
\end{lem}

\begin{thm}
The triangulated category $\cC_\Phi(Z_\infty^m)$ is $m+1$-Calabi-Yau, i.e., there is a natural isomorphism
\[
	\Ext^k(X,Y)\cong D\Ext^{m+1-k}(Y,X)
\]
for all objects $X,Y$.
\end{thm}

\begin{proof}
The lemma implies that the basic morphism $E(x,y)\to \t E(x,y)[1]=E(x,y)[m+1]$ is a \emph{hammock} which means that for any morphism $E(x,y)\to Y$ for any object $Y$, there is a morphism $Y\to E(x,y)[m+1]$. The duality is given by the fact that the nondegenerate pairing
\[
	(X,Y)\otimes (Y,X[m+1])\to (X,X[m+1])\cong \kk
\]
As in the proof of Theorem \ref{thm: C(Z) is 2-CY}, naturality of the pairing follows from the fact that it is symmetric up to sign.
\end{proof}

% subsection

%%\newpage
%-----------------------------------------------------------------------------------
%            sub section {Auslander-Reiten quiver of $\cC(Z_\infty^m)$}
%-----------------------------------------------------------------------------------

\subsubsection{Auslander-Reiten quiver of $\cC_\Phi(Z_\infty^m)$}

As in the previous section, we have the following.

\begin{thm}
The Auslander-Reiten quiver of $\cC_\Phi(Z_\infty^m)$ is a union of $\binom m2$ components $C_{pq}=C_{qp}$ where $p,q$ are distinct integers modulo $m$. Of these, the $m$ components $C_{p,p+1}$ are of type $\ZZ A_\infty$ and the others are of type $\ZZ A_\infty^\infty$. 
\end{thm}

\begin{proof}
The points in $Z_m\ast \ZZ$ form $m$ \emph{blocks} given by fixing the first coordinate. Objects $E(x,y)$ must have ends $x,y$ in different blocks. We let $C_{pq}$ be the set of all $E(x_i^p,x_j^q)$. These form a component of the Auslander-Reiten quiver since irreducible maps will change the subscripts by $1$. If $q=p+1$ then we must have $j>i$. So, $C_{p,p+1}$ is of type $\ZZ A_\infty$. The others have type $\ZZ A_\infty^\infty$ since $i,j$ are arbitrary.
\end{proof}

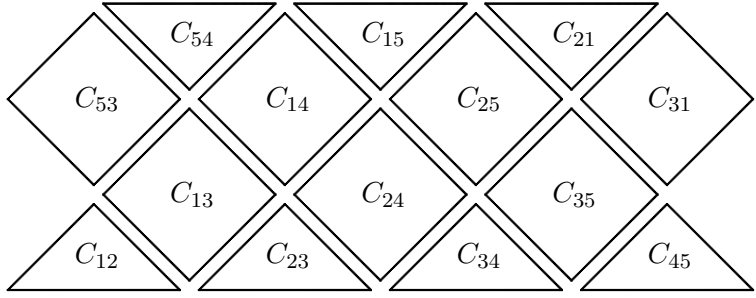
\begin{figure}[htbp]
\begin{center}
%
%\vs5
{
\setlength{\unitlength}{1in}
%\centerline
{\mbox{
\begin{picture}(4,1.5)
      \thicklines
%    \thinlines
%
\put(.5,1.5){\Ainftydown{$C_{54}$}} % top row
\put(1.5,1.5){\Ainftydown{$C_{15}$}}
\put(2.5,1.5){\Ainftydown{$C_{21}$}}
\put(0,1){\Ainftyinfty{$C_{53}$}} % second row
\put(1,1){\Ainftyinfty{$C_{14}$}} % second row
\put(2,1){\Ainftyinfty{$C_{25}$}} % second row
\put(3,1){\Ainftyinfty{$C_{31}$}} % second row
\put(.5,.5){\Ainftyinfty{$C_{13}$}} % third row
\put(1.5,.5){\Ainftyinfty{$C_{24}$}}
\put(2.5,.5){\Ainftyinfty{$C_{35}$}}
\put(0,0){\Ainftyup{$C_{12}$}} % bottom row
\put(1,0){\Ainftyup{$C_{23}$}}
\put(2,0){\Ainftyup{$C_{34}$}}
\put(3,0){\Ainftyup{$C_{45}$}}
\end{picture}}
}}
%\vs5
\caption{When $m=5$, there are $\binom m2=10$ components in the AR-quiver of $\cC_\Phi(Z_\infty^m)$. They are labeled $C_{pq}=C_{qp}$. When $p,q$ are consecutive, $C_{pq}$ has type $\ZZ A_\infty$. Otherwise $C_{pq}$ has type $\ZZ A_\infty^\infty$. Four components are repeated.}
\label{AR-quiver of Zm-infty}
\end{center}
\end{figure}
%

% subsection

%%\newpage
%-----------------------------------------------------------------------------------
%            sub section {Standard $m$-clusters}
%-----------------------------------------------------------------------------------

\subsubsection{Standard objects and $m$-clusters}

We define the \emph{standard objects} of $\cC_\Phi(Z_\infty^m)$ to be those objects all of whose indecomposable summands lie in the $\ZZ A_\infty$ components $C_{p,p+1}$ of the Auslander-Reiten quiver of $\cC_\Phi(Z_\infty^m)$. In other words, the standard indecomposable objects are $E(x_i^p,x_j^{p+1})$ where $i<j$.

The main properties of standard objects are given below. We recall that an indecomposable object $X$ in an $m+1$-Calabi-Yau triangulated category if \emph{$m+1$-rigid} if $\Ext^k(X,X)=0$ for $1\le k\le m$. Two such objects are \emph{compatible} if $\Ext^k(X,Y)=0$ for $1\le k\le m$.

\begin{thm}\label{thm: main thm on standard m-clusters} The full subcategory of $\cC_\Phi(Z_\infty^m)$ consisting of standard objects is a thick subcategory. This subcategory is a true \emph{$m$-cluster category} in the following sense.
\begin{enumerate}
\item All standard indecomposable objects $X$ are {$m+1$-rigid}.% in the sense that $\Ext^i(X,X)=0$ for $1\le i\le m$.
\item Define a \emph{standard $m$-cluster} to be a maximal compatible set of indecomposable standard objects of $\cC_\Phi(Z_\infty^m)$. Then for any such set $\cT$ and any object $T\in\cT$, there are exactly $m$ other standard objects $T_i^\ast$ up to isomorphism so that $\cT\backslash T\cup T_i^\ast$ is a standard $m$-cluster.
\item The objects $T_i^\ast$ with $T_0^\ast=T$ are uniquely determined up to isomorphism by the fact that there is a distinquished triangle
\[
	T_i^\ast\to B_i\to T_{i+1}^\ast\to T_i^\ast[1]
\]
where $B_i$ is the left $add\,\cT\backslash T$-approximation of $T_i^\ast$.
\end{enumerate}
Furthermore, the Verdier quotient $\cC(Z_\infty^m)/\cC_\infty^m$ is triangle equivalent to the standard cluster category of type $A_{m-3}$.
\end{thm}

If $X=E(x,y)$ then let $\ll_1X<\ll_2X$ be the integers $\ll(x),\ll(y)$ in increasing order and let $\ll X=(\ll_1X,\ll_2X)$. We say that $X$ \emph{crosses} $Y$ if $\ll_1X<\ll_1Y<\ll_2X<\ll_2Y$.

\begin{lem}\label{lem: standard objects are compatible iff noncrossing}
Two standard objects $X,Y$ are compatible if and only if they are noncrossing, i.e., they do not cross each other.
\end{lem}

Since $X$ never crosses itself, this implies that standard objects are $m+1$-rigid. Lemma \ref{lem: standard objects are compatible iff noncrossing} also implies the following theorem.

\begin{thm}\label{thm: standard m-clusters give partitions of infinity-gon}
Isomorphism classes of standard $m$-clusters are in 1-1 correspondence with the partitions of the $\infty$-gon into $m+2$-gons.
\end{thm}

\begin{defn}\label{defn: infinity-gon}
By the \emph{$\infty$-gon} we mean the convex hull $P_\infty$ of the set $V$ of all points on the unit circle in $\CC$ of the form $e^{i\th}$ where $\cot \frac\th2\in\ZZ$. This gives a bijection $V\cong\ZZ$. We call $V$ the \emph{vertex set} of $P_\infty$. By a \emph{partition} of $P_\infty$ we mean a set of noncrossing chords in the circle with both endpoints in the vertex set $V$. (See \cite{HJ12}.)
\end{defn}

\begin{proof} For any standard object $X=E(x,y)$, take the chord in the $\infty$-gon with endpoints corresponding to $\ll(x),\ll(y)\in\ZZ$. If $X,Y$ are compatible then the corresponding chords do not cross. Therefore, any standard $m$-cluster gives a partition of the $\infty$-gon $P_\infty$ into polygons. Take one of these polygons, say $P$ and let $(a,b)=\ll X$ be the longest side of $P$. Since $X$ is standard, $b-a\equiv 1$ mod $m$ and $b\ge a+m+1$. If $b=a+m+1$ then the other sides of $P$ are sides of $P_\infty$ and $P$ is an $m+2$-gon. Otherwise, take the shortest chord $C=(v,w)$ inside the closed interval $[a,b]$. Then $C$ has length $m+1$ and bounds an $m+2$-gon. Contract $C$ to the side $(v,v+1)$ on $P_\infty$ and subtract $m$ from the position of all vertices $\ge w$. Then $b-a$ decreases by $m$ and we conclude by induction on $b-a$ that $P$ is an $m+2$-gon. The converse argument is similar.
%Conversely, suppose that we have a partition of the $\infty$-gon into $m+2$-gons. Take a minimal chord in this partition. It must be $(a,b)$ where $b=a+m+1$. Then $(a,b)=\ll X$ for some standard object $X$. As before, we can subtract $m$ from the positions of all endpoints $\ge b$ and contract the chord $(a,b)$ to the side $(a,a+1)$ of $P_\infty$ to obtain another partition of $P_\infty$ into $m+2$-gons. By induction, the chords are all standard in the sense that their lengths are $\equiv 1$ mod $m$. But this implies the original chords were standard. So, the original partition was a standard $m$-cluster.
\end{proof}

To prove Lemma \ref{lem: standard objects are compatible iff noncrossing}, we examine which standard objects map to each other. %We give only the statements and no proofs since, when examined in order, each lemma is very straightforward.

\begin{lem}\label{lem: condition for Hom(X,Y) not zero X,Y standard}
Let $X,Y$ be two standard objects in the same component of the AR-quiver. Then $\cC_\Phi(X,Y)\neq 0$ if and only if
\begin{equation}\label{eq: condition for Hom(X,Y) not zero X,Y standard}
	\ll_1X\le \ll_1Y<\ll_2X-1\le \ll_2Y-1
\end{equation}
\end{lem}

The proof is straightforward and left to the reader. Note that these four integers are congruent to each other modulo $m$.

\begin{lem}\label{lem: X crosses Y iff not compatible}
Let $X,Y$ be standard objects. Then the following two statements are equivalent.
\begin{enumerate}
\item There exists $1\le k\le m$ so that $Y[k]$ lies in the same AR-component as $X$ and so that $\cC_\Phi(X,Y[k])\neq0$.
\item $X$ crosses $Y$, i.e., $\ll_1X<\ll_1Y<\ll_2X<\ll_2Y$.
\end{enumerate}
\end{lem}

\begin{proof} Without loss of generality we may assume that $\ll_1 X$ is divisible by $m$. Thus, when $Y[k]$ lies in the same AR-component as $X$,  $\ll_1Y[k]$ will also be divisible by $m$ and $\ll_1Y[k]/m=\lfloor (\ll_1 Y-1)/m\rfloor$. Similarly, $\ll_2Y[k]-1$ is divisible by $m$ and $(\ll_2Y[k]-1)/m=\lfloor (\ll_2 Y-2)/m\rfloor$. So, by \eqref{eq: condition for Hom(X,Y) not zero X,Y standard}, Condition (1) is equivalent to the integer inequality:
\begin{equation}\label{eq: first technical condition for compatibility}
	\frac{\ll_1 X}m\le \left\lfloor\frac{\ll_1Y-1}m\right\rfloor<\frac{\ll_2 X-1}m\le \left\lfloor\frac{\ll_2Y-2}m\right\rfloor
\end{equation}
For all integers $a,b$ and real numbers $x,y$, the integer inequality $a\le \lfloor x\rfloor<b\le \lfloor y\rfloor$ is equivalent to the real inequality $a\le x<b\le y$. So, \eqref{eq: first technical condition for compatibility} is equivalent to the condition
\[
	\ll_1 X\le \ll_1Y-1<\ll_2X-1\le \ll_2Y-2
\]
which is equivalent to Condition (2).
\end{proof}

\begin{lem}\label{lem: Y crosses X iff not compatible}
Let $X,Y$ be standard objects. Then the following two statements are equivalent.
\begin{enumerate}
\item There exists $1\le k\le m$ so that $\cC_\Phi(X,Y[k])\neq0$ but $Y[k]$ does not lie in the same AR-component as $X$.
\item $Y$ crosses $X$, i.e., $\ll_1Y<\ll_1X<\ll_2Y<\ll_2X$.
\end{enumerate}
\end{lem}

\begin{proof} 
Since $\cC_\Phi$ is $m+1$-CY, the condition $\cC_\Phi(X,Y[k])\neq0$ is equivalent to the condition that $\cC_\Phi(Y,X[m+1-k])\neq0$. In Condition (1) we must also have $\ll_1Y[k]=\ll_1Y-k\equiv\ll_1X-1$ modulo $m$. So, $\ll_1X[m+1-k]\equiv\ll_1X-1+k\equiv \ll_1Y$ mod $m$. Thus, Condition (1) is equivalent to the condition that there exists $1\le\el\le m$ so that $X[\el]$ lies in the same AR-component as $Y$ and $\cC_\Phi(Y,X[\el])\neq0$. By the previous lemma, this is equivalent to Condition (2).
\end{proof}

\begin{proof}[Proof of Lemma \ref{lem: standard objects are compatible iff noncrossing}] There are two ways that $X,Y$ might be not compatible. By
Lemmas \ref{lem: X crosses Y iff not compatible} and \ref{lem: Y crosses X iff not compatible} these correspond exactly to the two different ways that $X,Y$ can cross. Therefore, they are compatible iff they don't cross.
\end{proof}

\begin{lem}\label{morphisms between compatible standard objects}
Suppose $X,Y$ are nonisomorphic compatible standard objects in $\cC_\Phi(Z_\infty^m)$. Then $\cC_\Phi(X,Y)\neq0$ if and only if one of the following holds.
\begin{enumerate}
\item $\ll_1X=\ll_1 Y$ and $\ll_2X<\ll_2Y$.
\item $\ll_1X<\ll_1 Y$ and $\ll_2X=\ll_2Y$.
\item $\ll_1X=\ll_2Y$. 
\end{enumerate}
\end{lem}

\begin{proof}
If $X,Y$ are in the same AR-component, $\cC_\Phi(X,Y)\neq0$ is equivalent to (1) or (2) by Lemma \ref{lem: condition for Hom(X,Y) not zero X,Y standard} given that $X,Y$ are compatible and thus noncrossing. If $X,Y$ are not in the same AR-component then $\cC_\Phi(X,Y)\neq0$ is equivalent to $\cC_\Phi(Y,X[m+1])\neq0$ with $Y,X[m+1]$ being in the same AR-component. By Lemma \ref{lem: condition for Hom(X,Y) not zero X,Y standard} this is equivalent to:
\[
	\ll_1Y\le \ll_1X-m-1<\ll_2Y-1\le \ll_2X-m-2.
\]
Since these integers are congruent to each other modulo $m$, this is equivalent to:
\[
	\ll_1Y< \ll_1X-1\le\ll_2Y-1< \ll_2X-2,
\]
a condition equivalent to (3) given that $X,Y$ are noncrossing.
\end{proof}

\begin{proof}[Proof of Theorem \ref{thm: main thm on standard m-clusters}]
We are now ready to prove the main theorem about standard $m$-clusters. 

(1) We have already observed that standard objects are $m+1$-rigid since they do not cross themselves.

(2) Every element $T$ of every standard $m$-cluster $\cT$ represents a chord in $P_\infty$ which separates two $m+2$-gons. When the chord is removed, we have a $2m+2$-gon $G$ which has $m+1$ diagonals. So, there are $m$ other diagonals corresponding to $m$ other standard objects $T_i^\ast$ so that $\cT\backslash T\cup T_i^\ast$ forms a standard $m$-cluster.

(3) By Lemma \ref{morphisms between compatible standard objects}, morphisms between compatible standard objects correspond to counterclockwise pivots of the corresponding chords in $P_\infty$. So, the left $add\,\cT\backslash T$-approximation of $T_i^\ast$ is given by the direct sum of the two objects $B_i',B_i''$ (corresponding to chords) which form the sides of the $2m+2$-gon $G$ which are adjacent to $T_i^\ast$ and are clockwise from the endpoints of $T_i^\ast$ as shown in the figure. The sequence 
\[
	T_i^\ast\xrarrow{\mat{1\\-1}} B_i'\oplus B_i''\xrarrow{[1,1]} T_{i+1}^\ast
\]
is exact in $\cMF_\Phi(Z_\infty^m)$ and therefore forms a distinguished triangle in the stable category $\cC_\Phi(Z_\infty^m)$. Thus the objects $T_i^\ast$ form the diagonals of the $2m+2$-gon $G$ ordered clockwise.
\begin{figure}[htbp]
\begin{center}
%
%\vs5
{
\setlength{\unitlength}{.5in}
%\centerline
{\mbox{
\begin{picture}(3,2)
      \thicklines
%    \thinlines
\put(0,.5){\qbezier(0,0)(1.5,.5)(3,1)}
\put(2.1,1.4){$T_i^\ast$}
\put(.6,1.4){$T_{i+1}^\ast$}
\put(0,1.5){\qbezier(0,0)(1.5,-.5)(3,-1)}
\put(0,.5){\line(0,1)1}
\put(-.5,.9){$B_i'$}
\put(3,.5){\line(0,1)1}
\put(3.15,.9){$B_i''$}
\put(0,1.5){\qbezier(0,0)(.4,.25)(.8,.5)}
\put(0,.5){\qbezier(0,0)(.4,-.25)(.8,-.5)}
\put(3,1.5){\qbezier(0,0)(-.4,.25)(-.8,.5)}
\put(3,.5){\qbezier(0,0)(-.4,-.25)(-.8,-.5)}
\end{picture}}
}}
%\vs5
\caption{The objects $B_i'$ and $B_i''$ share an endpoint with $T_i^\ast$ and are the objects in $\cT\backslash T$ closest to $T_i^\ast$ in the counterclockwise direction. So, their sum $B_i=B_i'\oplus B_i''$ form the left $\cT\backslash T$-approximation to $T_i^\ast$ and $T_i^\ast\to B_i'\oplus B_i''\to T_{i+1}^\ast$ is a distinguished triangle in $\cC_\Phi(Z_\infty^m)$.}
\label{Figure3}
\end{center}
\end{figure}

To show that the standard objects form a thick subcategory $\cC_\infty^m$ of $\cC_\Phi(Z_\infty^m)$, it suffices to show that $\cC_\infty^m$ is a triangulated full subcategory, in other words, for any morphism of standard objects $f:X\to Y$, the third object in the distinguished triangle $X\to Y\to Z\to X[1]$ is also standard. By the octahedral axiom it suffices to consider the case when $X$ is indecomposable. So, suppose that $X\in C_{p,p+1}$. Then, dropping all of the components of the map $f:X\to Y$ which are zero, we may assume that $Y$ lies in $C_{p,p+1}\cup C_{p-1,p}$. Since $X[1]$ lies in $C_{p-1,p}$, all of the endpoints of the object $Z$ must lie in blocks $p-1,p,p+1$. Therefore, the only possible nonstandard components of $Z$ lie in $C_{p-1,p}$. But this is impossible since none of these objects map nontrivially to $X[1]\in C_{p-1,p}$. So, all components of $Z$ are standard. %and $\cC_\infty^m$ is thick.

Finally, to identify the Verdier quotient category $\cC_\Phi(Z_\infty^m)/\cC_\infty^m$, we claim that the objects of this quotient are the components of the AR-quiver of $\cC_\Phi(Z_\infty^m)$. This follows from two considerations:
\begin{enumerate}
\item Any two indecomposable objects in the same component $C_{pq}$ are equivalent in $\cC_\Phi(Z_\infty^m)/\cC_\infty^m$.
\item There is an exact morphism of Frobenius categories $\cMF_\Phi(Z_\infty^m)\to \cMF_\phi(Z_m)$ so that the inverse image of the indecomposable objects of $\cMF_\phi(Z_m)$ are subcategories of $\cMF_\Phi(Z_\infty^m)$ which map onto the AR-components of $\cC_\Phi(Z_\infty^m)$ and the inverse image of the projective-injective objects are exactly the standard objects (and the projective-injective objects) of $\cMF_\Phi(Z_\infty^m)$. 
\end{enumerate}
The first statement is easy to see. For example, there is an exact sequence \[
E(x_i^p,x_j^q)\to E(x_k^p,x_j^q)\oplus E(x_i^p,x_N^{p+1})\to E(x_k^p,x_N^{p+1})
\]
for sufficiently large $N$ showing that $E(x_i^p,x_j^q)\sim E(x_k^p,x_j^q)$ modulo standard objects and, similarly, $E(x_i^p,x_j^q)\sim E(x_i^p,x_k^q)$.

The exact morphism in the second statement is induced by the morphism of cyclic posets $p:Z_m\ast \ZZ\to Z_m$ given by projection to the first coordinate. This morphism is compatible with $\Phi$ and $\phi$ since $\phi\circ p=p\circ\Phi$. Therefore, it induces an exact functor $\cMF_\Phi(Z_\infty^m)\to \cMF_\phi(Z_m)$ sending projective-injective objects to projective injective objects and it is easy to see that it has the stated properties.
\end{proof}

We saw that the endpoints of a standard object lie in two consecutive \emph{blocks}: $x^p_\ast$ and $x^{p+1}_\ast$. We now consider nonstandard objects (with blocks a distance of at least 2 apart).

% subsection

%%\newpage
%-----------------------------------------------------------------------------------
%            sub section {Nonstandard objects periodic $m$-clusters}
%-----------------------------------------------------------------------------------

\subsubsection{Nonstandard objects and periodic $m$-clusters}

\begin{prop}
An indecomposable object $X$ of $\cC_\Phi(Z_\infty^m)$ is $m+1$-rigid iff it is one of the following.
\begin{enumerate}
\item $X$ is standard.
\item $X=E(x^p_i,x^{p+2}_j)$ where $i>j$.
\end{enumerate}
\end{prop}

\begin{rem}
We call objects of the second kind \emph{nonstandard $m+1$-rigid objects}. Since $\ll(x^p_i)>\ll(x^{p+2}_j)$, the nonstandard $m+1$ rigid objects can also be described as those objects $Y$ where $\ll_1Y-\ll_2Y\equiv 2$ mod $m$.
\end{rem}

\begin{proof}
Suppose first that the endpoints of $X$ lie in blocks at least 3 apart: $X=E(x^p_i,x^q_j)$ where $p+3\le q\le p+m-3$. Then $X[m-1]=E(x^{p+1}_{i-1},x^{q+1}_{j-1})$ and $\cC_\Phi(X,X[m-1])\neq0$. So, $X$ is not rigid. Next, suppose that $X=E(x^p_i,x^{p+2}_j)$. Then the only possible self-extension of $X$ is $\cC_\Phi(X,X[m-1])$. Since $X[m-1]=E(x^{p+1}_{i-1},x^{p+3}_{j-1})$ we have $\cC_\Phi(X,X[m-1])\neq0$ iff $i-1<j$. So we need $j<i$ for $X=E(x^p_i,x^{p+2}_j)$ to be rigid.
\end{proof}

\begin{cor}
All objects of $\cC_\Phi(Z_\infty^m)$ are $m+1$ rigid iff $m\le 4$.
\end{cor}

Define a \emph{nonstandard $m$-cluster} to be a maximal collection of pairwise compatible $m+1$-rigid indecomposable objects of $\cC_\Phi(Z_\infty^m)$ containing at least one nonstandard object.

Let $B=\RR\times[0,1]$ and let $\d B^{(2)}$ be the set of all two element subsets of the boundary $\d B=\RR\times\{0,1\}$. Let $\ll_0,\ll_1:Z_\infty^m\to \d B^{(2)}$ be the two embeddings given by $\ll_0x=(\ll(x),0)$, $\ll_1x=(-\ll(x),1)$.

%noncrossing lifting of the bijection $Z_\infty^m\to \tilde C$ (i.e., for distinct elements $w,x,y,z\in Z_\infty^m$ the pairs $(\psi_0 w,\psi_0x), (\psi_1y,\psi_1z)$ are noncrossing). 

\begin{thm}\label{thm: nonstandard m-clusters and 2-periodic decompositions}
Let $\Psi$ be the mapping which sends any standard object $E(x,y)$ to the pair of pairs $\Psi E(x,y)=\{(\ll_0x,\ll_0y),(\ll_1x,\ll_1y)\}$ and any nonstandard $m+1$-rigid object $E(x,z)$ to $\Psi E(x,z)=\{(\ll_0x,\ll_1z),(\ll_1x,\ll_0z)\}$. Then a collection of rigid indecomposable objects of $\cC_\Phi(Z_\infty^m)$ form a nonstandard $m$-cluster iff it is mapped by $\Psi$ to a 2-periodic decomposition of the $A_\infty^2$-gon into $m+2$-gons (except for the one in the middle).
\end{thm}

The proof follows the same pattern as the proof of Theorem \ref{thm: standard m-clusters give partitions of infinity-gon}. We state the corresponding sequence of lemmas without proof. The central polygon is discussed in Proposition \ref{central polygon in nonstandard m-cluster}.

\begin{lem}
Suppose that $X=E(x_a^p,x_b^{p+1})$ and $Y=E(x^{p-1}_i,x^{p+1}_j)$ with $i>j$. Then the following are equivalent.
\begin{enumerate}
\item $\cC_\Phi(Y,X)\neq0$
\item $a<j\le b$
\item $\ll_1X<\ll_1Y-1\le\ll_2X-1$
\end{enumerate}
\end{lem}

\begin{lem}
Suppose that $X=E(x_a^{p-2},x_b^{p-1})$ and $Y=E(x^{p-1}_i,x^{p+1}_j)$ with $i>j$. Then the following are equivalent.
\begin{enumerate}
\item $\cC_\Phi(Y,X)\neq0$
\item $a<i\le b$
\item $\ll_1X+1<\ll_2Y\le\ll_2X$
\end{enumerate}
\end{lem}

In both lemmas, the three integers in (3) are congruent to each other modulo $m$. The condition $i>j$ is equivalent to $Y$ being $m+1$-rigid when $m\ge5$.

\begin{lem} Let $Y=E(x^{p-1}_i,x^{p+1}_j)$ with $i>j$ and let $X$ be a standard indecomposable object. 
\begin{enumerate}
\item There exists $1\le k\le m$ so that $X[k]\in C_{p,p+1}$ and $\Ext^k(Y,X)\neq0$ if and only if
\[
\ll_1X<\ll_1Y<\ll_2X.\]
\item There exists $1\le k\le m$ so that $X[k]\notin C_{p,p+1}$ and $\Ext^k(Y,X)\neq0$ if and only if
\[
\ll_1X<\ll_2Y<\ll_2X.\]
\end{enumerate} 
\end{lem}

\begin{prop}\label{prop: compatibility of standard with nonstandard}
Suppose that $X,Y$ are $m+1$-rigid indecomposable objects where $X$ is standard and $Y$ is nonstandard. Then $X,Y$ are compatible if and only if neither $\ll_1Y$ not $\ll_2Y$ lies between $\ll_1X$ and $\ll_2X$. Equivalently, $X,Y$ are compatible if and only if $\Psi X,\Psi Y$ do not cross.
\end{prop}

\begin{lem}
Suppose that $Y=E(x_i^{p-1},x_j^{p+1}),Z=E(x_a^{q-1},x_b^{q+1})$ where $i>j$ and $a>b$.
\begin{enumerate}
\item Suppose $q=p$. Then $\cC_\Phi(Y,Z)\neq0$ iff $\ll_1Y\le \ll_1Z$ and $\ll_2Y\le \ll_2Z$.
\item Suppose $q=p+1$. Then $\cC_\Phi(Y,Z)\neq0$ iff $\ll_2Z< \ll_1Y-1$.
\end{enumerate} 
\end{lem}

%In the next lemma we assume $m\ge5$. The analogous statement hold for $m=4$. It is just difficult to state.

\begin{lem}
Suppose $m\ge5$ and $Y,Z$ are nonstandard $m+1$-rigid indecomposable objects with $Y\in C_{p-1,p+1}$. Then there exists $1\le k\le m$ so that $\cC_\Phi(Y,Z[k])\neq0$ and $Z[k]\in C_{p-1,p+1}\cup C_{p,p+2}$ if and only if either $\ll_2Y<\ll_2Z$ and $\ll_1Y<\ll_1Z$ or $\ll_2Z<\ll_1Y$.
\end{lem}

The same statement holds for $m=4$ if $C_{p-1,p+1}$ is construed to mean the set of all $E(x^{p-1}_i,x^{p+1}_j)$ with $i>j$ and similarly for $C_{p,p+2}$.

Using Serre duality to get the other two cases (when $Z[k]\in C_{p-2,p}\cup C_{p-3,p-1}$), we get that nonstandard $Y,Z$ are compatible if and only if all of the following hold:
\begin{enumerate}
\item $\ll_2Y\ge \ll_2Z$ or $\ll_1Y\ge \ll_1Z$
\item $\ll_2Z\ge \ll_1Y$
\item $\ll_2Z\ge \ll_2Y$ or $\ll_1Z\ge \ll_1Y$
\item $\ll_2Y\ge\ll_1Z$
\end{enumerate}
By manipulating these inequalities we get the following.

\begin{prop}\label{prop: compatibility of two nonstandard objects}
Suppose $Y,Z$ are nonstandard $m+1$-rigid indecomposable objects. Then $Y,Z$ are compatible if and only if one of the intervals $[\ll_1Y,\ll_2Y],[\ll_1Z,\ll_2Z]$ contains the other. Equivalently, $Y,Z$ are compatible if and only if $\Psi Y,\Psi Z$ do not cross.
\end{prop}

Theorem \ref{thm: nonstandard m-clusters and 2-periodic decompositions} follows from Propositions \ref{prop: compatibility of standard with nonstandard} and \ref{prop: compatibility of two nonstandard objects}.

\begin{prop}\label{central polygon in nonstandard m-cluster}
The central polygon in nonstandard $m$-cluster has $2m-2$ sides. Given any nonstandard $m$-cluster $\cT$ and object $T\in\cT$, if $T$ is not a side of the central $2m-2$-gon, then there are, up to isomorphism, exactly $m$ objects $T^\ast$ so that $\cT\backslash T\cup T^\ast$ forms a nonstandard $m$-cluster. When $T$ is a side of the central $2m-2$-gon then there are, up to isomorphism, exactly $3m-6$ objects $T^\ast$ so that $\cT\backslash T\cup T^\ast$ forms a nonstandard $m$-cluster. Furthermore, in both cases the objects $T^\ast$ are obtained as a sequence $T_i^\ast$ given by left $\cT\backslash T$-approximation triangles $T_i^\ast\to B_i\to T_{i+1}^\ast$.
\end{prop}

\begin{proof}
As in the proof of Theorem \ref{thm: standard m-clusters give partitions of infinity-gon}, we can contract all of the standard objects without changing the number of sides in the polygons having nonstandard sides. Then the central polygon has only two nonzero sides which are given by a single doubled nonstandard object $Y=E(x_i^p,x_j^{p+2})$. We must have $i=j+1$ (otherwise, $E(x_i^p,x_{i-1}^{p+2})$ would be compatible with all objects in $\cT$). Then $\ll_2Y-\ll_1Y=m-2$ which means that the chord connecting the endpoints in $P_\infty$ has $m-1$ sides which means that the doubled object in $P_\infty^{(2)}$ has $2(m-1)$ sides.

When a side of this central $2m-2$-gon is removed then two sides are removed by symmetry. So, we get a $6m-10$-gon. which has $3m-5$ trisections. So, there are $3m-6$ mutations of any side. Figure \ref{Figure3} applies to this case and shows that the objects $T_i^\ast$ can be obtained by iterated approximation triangles.
\end{proof}

\subsubsection{Example}
For $m=5$, here is an example of a nonstandard $m$-cluster. Stardard objects are pairs of horizontal arc connecting $x_i^p$ to $x_j^{p+1}$ (given by consecutive letters in the notation of the figure). Nonstandard objects are pairs of vertical arcs connecting connecting $x_i^{p+2}$ on one side to $x_j^p$ on the other with $i<j$.

There is $2m-2=8$-gon in center. Other regions have $m+2=7$ sides.
\[
\xymatrixrowsep{7pt}\xymatrixcolsep{10pt}
\xymatrix{%begin xy matrix
1&1&1&1&1&0&0&0&0&0&-1&-1&-1&-1&-1\\
E\ar@{-}[r]&D\ar@{-}[r]&C\ar@{-}[r]&B\ar@{-}[r]&A\ar@{-}[r]&E\ar@{-}[r]&D\ar@{-}[r]&C\ar@{-}[r]&B\ar@{-}[r]&A\ar@/^2.5pc/[llllll]^(.7){X_1}_{7\text{-gon}}\ar@{-}[r]&E\ar@{-}[r]&D\ar@{-}[r]&C\ar@{-}[r]&B\ar@{-}[r]&A&\\
\\
\\
&&&&&&& 8\text{-gon}\\
\\
\\
A\ar@{-}[ruuuuuu]^{Y_2}_{7 {\text{-gon}}}\ar@{-}[r]&B\ar@{-}[r]&C\ar@{-}[r]&D\ar@{-}[r]&E\ar@{-}[lluuuuuu]_{Y_1}\ar@{-}[r]&A\ar@/^2.5pc/[rrrrrr]^(.7){X_1}_{7\text{-gon}}\ar@{-}[r]&B\ar@{-}[r]&C\ar@{-}[r]&D\ar@{-}[r]&E\ar@{-}[r]&A\ar@{-}[r]&B\ar@{-}[r]&C\ar@{-}[lluuuuuu]^{Y_1}\ar@{-}[r]&D\ar@{-}[ruuuuuu]_{Y_2}^{7\text{-gon}}\ar@{-}[r]&E&\\
-1&-1&-1&-1&-1&0&0&0&0&0&1&1&1&1&1&\\
	}%end xy matrix
\]

Standard: $X_1=E(A_0,B_1)=E(x_0^1,x_1^2)$ (horizontal).

$Y_1=E(C_1,E_{-1})=E(x_1^3,x_{-1}^5)$, $Y_2=E(D_1,A_{-1})=E(x_1^4,x_{-1}^1)$ are nonstandard but $(m+1)$-rigid (vertical).

$x_j^p$ is the $p$th letter of the alphabet with subscript $j$. For example, $x_j^1=A_j, x_j^2=B_j$.

We want to emphasize that the thick subcategory of standard objects in $\cC_\Phi(Z_\infty^m)$ forms an $m$-cluster category in the usual sense. This example is illustrating the properties of the nonstandard objects in our category.

%%%%%%%%%%%%%%%%%%%%%%%%%%%%%%%%%%%%
%
%		BIBLIOGRAPHY
%
%%%%%%%%%%%%%%%%%%%%%%%%%%%%%%%%%%%%

%%%%%%%%%%%%%%%%%%%%%%%%%%%%%%%%%%%
\end{document}